\documentclass{amsart}
\usepackage[utf8]{inputenc}
\usepackage{amsmath}
\usepackage{amsfonts}
\usepackage{amssymb}
\usepackage{amsthm}
\newtheorem{thm}{Theorem}[section]
\newtheorem{cor}[thm]{Corollary}

\newtheorem{lem}[thm]{Lemma}
\newtheorem{prop}[thm]{Proposition}
\newtheorem{rem}[thm]{Remark}

\newcommand{\Real}{{\mathbb{R}}}
\newcommand{\NF}{{\mathbb{L}}}
\newcommand{\field}{{\mathbb{K}}}
\newcommand{\Com}{{\mathbb{C}}}
\newcommand{\Zet}{{\mathbb{Z}}}
\newcommand{\Nat}{{\mathbb{Z}_{\geq 0}}}
\newcommand{\Pos}{{\mathbb{Z}_{>0}}}
\newcommand{\Pro}{{\mathbb{P}}}

\newcommand{\Res}{{\mathrm{Res}}}

\newcommand{\ab}{{\mathbf{a}}}
\newcommand{\bb}{{\mathbf{b}}}

\newcommand{\jb}{{\mathbf{j}}}
\newcommand{\ori}{{\mathbf{0}}}
\newcommand{\LT}{{\mathrm{LT}}}
\newcommand{\LC}{{\mathrm{LC}}}
\newcommand{\Dom}{{\mathrm{Dom}}}

\newcommand{\Cc}{{\mathcal{C}}}
\newcommand{\Zc}{{\mathcal{Z}}}
\newcommand{\Pc}{{\mathcal{P}}}
\newcommand{\Fc}{{\mathcal{F}}}
\newcommand{\Gc}{{\mathcal{G}}}
\newcommand{\fu}{{\hat{h}}}
\newcommand{\lu}{{l_1}}
\newcommand{\gu}{{g}}
\newcommand{\fd}{{\dot{h}}}
\newcommand{\ld}{{\dot{l_1}}}
\newcommand{\ldd}{{\dot{l_2}}}
\newcommand{\gd}{{\dot{g}}}
\newcommand{\Mat}{{\mathrm{M_{2\times 2}}}}
\newcommand{\GL}{{\mathrm{GL_{2}}}}
\begin{document}
\title{Elekes-R\'{o}nyai Theorem revisited}
\author{Mario Huicochea\\CONACyT-UAZ}
\date{}
\begin{abstract}
  In this paper it is proven that for any $f\in\Real(x_1,x_2)$ and $A_1,A_2$ nonempty finite subsets  of $\Real$ such that $|A_1|=|A_2|$ and $f$ is defined in $A_1\times A_2$, we have that
  \begin{equation*}
  |f(A_1,A_2)|=\Omega\left(|A_1|^{\frac{4}{3}}\right)
  \end{equation*}
  unless there are $g,l_1,l_2\in\Real(x)$ such that $f$ satisfies one of the following equalities:
  \begin{enumerate}
  \item[$\bullet$]$f(x_1,x_2)=g(l_1(x_1)+l_2(x_2))$
  \item[$\bullet$]$f(x_1,x_2)=g(l_1(x_1)\cdot l_2(x_2))$
  \item[$\bullet$]$f(x_1,x_2)=g\left(\frac{l_1(x_1)+l_2(x_2)}{1-l_1(x_1)\cdot l_2(x_2)}\right)$.
  \end{enumerate}
  This result improves Elekes-R\'{o}nyai Theorem and it generalizes a result  of Raz-Sharir-Solymosi proven for $f\in\Real[x_1,x_2]$. Furthermore, an analogous result is proven for $f\in\Com(x_1,x_2)$ and $A_1,A_2$ subsets of $\Com$. 
  \end{abstract}
\maketitle

\section{Introduction}
In this paper we denote by $\Com,\Real,\Zet,\Pos$ and $\Nat$ the set of complex numbers, real numbers, integers, positive integers and nonnegative integers, respectively. For any field $\field$ and $x_1,x_2,\ldots, x_n$ algebraically independent variables over $\field$, we denote by $\field[x_1,x_2,\ldots, x_n]$  the set of polynomials in $x_1,x_2,\ldots, x_n$ with coefficients in $\field$. The quotient field of $\field[x_1,x_2,\ldots, x_n]$ will be denoted by $\field(x_1,x_2,\ldots, x_n)$ and its elements will be known as \emph{rational functions}. Let $p,q\in\field[x_1,x_2,\ldots, x_n]$. If $p$ and $q$ do not have common nonunit factors, we write $(p,q)=1$. For any $i\in\{1,2,\ldots, n\}$, denote by $\deg_{x_i}p$ the degree of $p$ with respect to $x_i$; the total degree of $p$ will be denoted by $\deg p$.  For any $f\in\field(x_1,x_2,\ldots, x_n)$ and $i\in\{1,2,\ldots, n\}$, if $f=\frac{p}{q}$ with $p,q\in\field[x_1,x_2,\ldots, x_n]$ and   $(p,q)=1$, then we write $\deg_{x_i} f:=\max\{\deg_{x_i}p,\deg_{x_i}q\}$ and   $\deg f:=\max\{\deg p,\deg q\}$. For any $f\in \field(x_1,x_2,\ldots, x_n)$ with $f=\frac{p}{q}$ where $p,q\in\field[x_1,x_2,\ldots, x_n]$ and   $(p,q)=1$, set $\Dom f:=\{\ab\in\field^n:\;q(\ab)\neq 0\}$.

Let $A_1$ and $A_2$ be nonempty finite subsets of $\Real$. If $A_1$ and $A_2$ are arithmetic (resp. geometric) progressions with the same common difference (resp. ratio), then $|A_1+A_2|=|A_1|+|A_2|-1$ (resp. $|A_1\cdot A_2|=|A_1|+|A_2|-1$). Therefore if $|A_1|=|A_2|$, and there are $g,l_1,l_2\in\Real[x]$ such that $l_1(A_1)$ and $l_2(A_2)$ are arithmetic (resp. geometric) progressions with the same common difference (resp. ratio), then $g(l_1(A_1)+l_2(A_2))=O(|A_1|)$ (resp. $g(l_1(A_1)\cdot l_2(A_2))=O(|A_1|)$). In \cite[Conj. 1]{Ele}, G. Elekes asked if $g(l_1(x_1)+l_2(x_2))$ and  $g(l_1(x_1)\cdot l_2(x_2))$ were the only families of polynomials $f$ in $\Real[x_1,x_2]$ having the property that $|f(A_1,A_2)|=O(|A_1|)$ for any $A_1,A_2$ subsets of $\Real$ such that $|A_1|=|A_2|$. Later Elekes and L. R\'{o}nyai were able to prove something stronger than Elekes' Conjecture.
\begin{thm}
\label{R1}
For any $c\in\Real$ and $d\in\Pos$ with $c\geq 1$, there is $c_1=c_{1,c,d}$ with the following property. For any  $f\in\Real(x_1,x_2)$ such that  $\deg f\leq d$, if there are $A_1$ and $A_2$  nonempty finite subsets of $\Real$ with $|A_1|=|A_2|>c_1$ such that $A_1\times A_2\subseteq \Dom f$ and $|f(A_1,A_2)|\leq c|A_1|$, then there are $g,l_1,l_2\in\Real(x)$ such that $f$ satisfies one of the following equalities.
  \begin{enumerate}
  \item[i)]$f(x_1,x_2)=g(l_1(x_1)+l_2(x_2))$.
  \item[ii)]$f(x_1,x_2)=g(l_1(x_1)\cdot l_2(x_2))$.
  \item[iii)]$f(x_1,x_2)=g\left(\frac{l_1(x_1)+l_2(x_2)}{1-l_1(x_1)\cdot l_2(x_2)}\right)$.
  \end{enumerate}
\end{thm}
\begin{proof}
See \cite[Thm. 2]{ER}.
\end{proof}
Therefore Elekes and R\'{o}nyai proved that $|f(A_1,A_2)|$ is superlinear unless $f$ has one of the  specific forms i)-iii) of Theorem \ref{R1}. Later,  O. E. Raz,  M. Sharir and J. Solymosi showed that if $f\in\Real[x_1,x_2]$, then  $|f(A_1,A_2)|=\Omega\left(|A_1|^{\frac{4}{3}}\right)$ unless $f$ has a very special form; more specifically they showed the following.

\begin{thm}
\label{R2}
Let $d\in\Nat$, $f\in\Real[x_1,x_2]$ be such that $\deg f\leq d$ and $A_1,A_2$ be nonempty finite subsets of $\Real$ such that $|A_1|=|A_2|$. Then
\begin{equation*}
|f(A_1,A_2)|=\Omega_d\left(|A_1|^{\frac{4}{3}}\right)
\end{equation*}
unless there are $g,l_1,l_2\in\Real[x]$ such that $f$ satisfies one of the following equalities.
  \begin{enumerate}
  \item[i)]$f(x_1,x_2)=g(l_1(x_1)+l_2(x_2))$.
  \item[ii)]$f(x_1,x_2)=g(l_1(x_1)\cdot l_2(x_2))$.
    \end{enumerate}
\end{thm}
\begin{proof}
See \cite[Cor. 3]{RSS}.
\end{proof}
Moreover, similar results have been obtained in $\Com$, see \cite{RSZ}, \cite{SZ}. The previous results have been generalized and several applications have been found, see \cite{Zee} for a nice survey in this topic. However, as it is already noted  by F. de Zeeuw in \cite[Prob. 1.2]{Zee}, it remained an open problem to show Theorem \ref{R2} when $f$ is a rational function. To achieve this goal, the main feature  is the following result which we think it is interesting on its own.
\begin{thm}
\label{R3}
Let $\field\in\{\Real,\Com\}$,  $d\in\Nat$, $f\in\field(x_1,x_2)$ be such that $\deg f\leq d$ and $A_1,A_2$ be nonempty finite subsets of $\field$ such that $|A_1|=|A_2|$ and $A_1\times A_2\subseteq \Dom f$. Then
\begin{equation*}
|f(A_1,A_2)|=\Omega_d\left(|A_1|^{\frac{4}{3}}\right)
\end{equation*}
unless there are $g,l_1,l_2\in\field(x)$  and $h\in\field(x_1,x_2)$ such that $\deg_{x_1}h,\deg_{x_2}h\leq 1$ and 
\begin{equation*}
f(x,y)=g(h(l_1(x_1),l_2(x_2))).
\end{equation*}
\end{thm} 
Using Theorem \ref{R3} and other ideas, we are able to prove the main results of this paper. 
\begin{thm}
\label{R4}
Let  $d\in\Nat$, $f\in\Real(x_1,x_2)$ be such that $\deg f\leq d$ and $A_1,A_2$ be nonempty finite subsets of $\field$ such that $|A_1|=|A_2|$ and $A_1\times A_2\subseteq \Dom f$. Then
\begin{equation*}
|f(A_1,A_2)|=\Omega_d\left(|A_1|^{\frac{4}{3}}\right)
\end{equation*}
unless there are $g,l_1,l_2\in\Real(x)$ such that $f$ satisfies one of the following equalities.
  \begin{enumerate}
  \item[i)]$f(x_1,x_2)=g(l_1(x_1)+l_2(x_2))$.
  \item[ii)]$f(x_1,x_2)=g(l_1(x_1)\cdot l_2(x_2))$.
  \item[iii)]$f(x_1,x_2)=g\left(\frac{l_1(x_1)+l_2(x_2)}{1-l_1(x_1)\cdot l_2(x_2)}\right)$.
  \end{enumerate}
\end{thm}
Since $\Com$ is algebraically closed and $\Real$ is not, the case iii) of Theorem \ref{R4} cannot appear in the analogous result for the complex numbers. 
\begin{thm}
\label{R5}
Let  $d\in\Nat$, $f\in\Com(x_1,x_2)$ be such that $\deg f\leq d$ and $A_1,A_2$ be nonempty finite subsets of $\Com$ such that $|A_1|=|A_2|$ and $A_1\times A_2\subseteq \Dom f$. Then
\begin{equation*}
|f(A_1,A_2)|=\Omega_d\left(|A_1|^{\frac{4}{3}}\right)
\end{equation*}
unless there are $g,l_1,l_2\in\Com(x)$ such that $f$ satisfies one of the following equalities.
  \begin{enumerate}
  \item[i)]$f(x_1,x_2)=g(l_1(x_1)+l_2(x_2))$.
  \item[ii)]$f(x_1,x_2)=g(l_1(x_1)\cdot l_2(x_2))$.
   \end{enumerate}
\end{thm}

We sketch the proofs of the main results of this paper. First of all we must say that a lot of the ideas used in this paper were taken from Elekes-R\'{o}nyai's paper \cite{ER} and Raz-Sharir-Solymosi's paper \cite{RSS}.  We start with Theorem \ref{R3}. The key ingredient in the proof of Theorem \ref{R3} is Lemma \ref{R35}. This lemma roughly says that, with the assumptions and notation of Theorem \ref{R3}, one of the following claims must  be true:
 \begin{enumerate}
    \item[i)]$|f(A_1,A_2)|=\Omega_d\left(|A_1|^{\frac{4}{3}}\right)$.
    \item[ii)]There are $g,l_1\in\field(x)$ and $h\in\field(x_1,x_2)$ such that $\deg_{x_1}h\leq 1$ and $f(x_1,x_2)=g(h(l_1(x_1),x_2))$.
    \item[iii)]There are $g,l_2\in\field(x)$ and $h\in\field(x_1,x_2)$ such that $\deg_{x_2}h\leq 1$ and $f(x_1,x_2)=g(h(x_1,l_2(x_2)))$.
    \end{enumerate}
   We overview the proof of Lemma \ref{R35}.   Using some ideas of Elekes and Ronyai, see Lemma \ref{R33},  we can prove that $f(x_1,x_2)=g(\fu(x_1,x_2))$ for some $\fu\in\field(x_1,x_2)$ and $g\in\field(x)$ satisfying that for big subsets $B_1$ of $A_1\times A_1$ and $B_2$ of $A_2\times A_2$, we have  that for any pair $(a,a')\in B_1$ (resp. $(a,a')\in B_2$) the pair $\{f(a,x_2),f(a',x_2)\}$ (resp.  $\{f(x_1,a),f(x_1,a')\}$) cannot be factorized through a rational function $g'$ with $\deg g'>\deg g$. For any $\ab_2=(a_2,a'_2)\in A_2\times A_2$ (resp. $\ab_1=(a_1,a'_1)\in A_1\times A_1$), let $C^{(1)}_{\fu,\ab_2}$ (resp. $C^{(2)}_{\fu,\ab_1}$) be the  curve in $\field^2$ defined by the equation $\fu(x,a_2)-\fu(y,a'_2)=0$ (resp. $\fu(a_1,x)-\fu(a'_1,y)=0$), see Section 2 for the precise definitions. The proof of Lemma \ref{R35} depends on whether there exists an irreducible algebraic curve contained in  many curves of $\left\{C^{(1)}_{\fu,\ab_2}\right\}_{\ab_2\in B_2}$  or $\left\{C^{(2)}_{\fu,\ab_1}\right\}_{\ab_1\in B_1}$, or there is not such irreducible curve. If there is not an irreducible curve as above, then we use some ideas of Raz, Sharir and Solymosi  to conclude i) (this is done applying a Szemer\'{e}di-Trotter type result based on  a statement of Solymosi and de Zeeuw, see Theorem \ref{R23}). If there is an irreducible  curve as above, then we conclude (using some applications of Elimination Theory and L\"{u}roth's Theorem) that ii) or iii) is true.   The proof of Theorem \ref{R3} is similar to the proof of Lemma \ref{R35}. If $f$ satisfies i) in Lemma \ref{R35}, we are done. If $f$ satisfies ii) (the case iii) is done symmetrically), then $f(x_1,x_2)=g(\fu(l_1(x_1),x_2))$ for some $g,l_1\in\field(x)$ and  $\fu(x_1,x_2)$ with $\deg_{x_1} \fu\leq 1$. Thus, in this case as a consequence of Lemma \ref{R17} and some trivial reductions, most of the algebraic curves $\left\{C^{(1)}_{\fu,\ab_2}\right\}_{\ab_2\in A_2\times A_2}$  are irreducible and thus it remains to conclude the proof of Theorem \ref{R3} depending on whether there exists an irreducible algebraic curve contained in  many curves of $\left\{C^{(2)}_{\fu,\ab_1}\right\}_{\ab_1\in A_1\times A_1}$ or there is not such irreducible curve. This is done as in Lemma \ref{R35}, and then the proof of Theorem \ref{R3} is completed. The proofs of Theorem \ref{R4} and Theorem \ref{R5} are almost the same so we just sketch the one of Theorem \ref{R4}. The proof of Theorem \ref{R4} relies in Theorem \ref{R3}. If $|f(A_1,A_2)|=\Omega_d\left(|A_1|^{\frac{4}{3}}\right)$, there is nothing to prove; assume that this is not the case. Then  $f(x_1,x_2)=\gd(\fd(\ld(x_1),\ldd(x_2)))$ for some $\gd,\ld,\ldd\in\field(x)$ and  $\fd(x_1,x_2)$ with $\deg_{x_1} \fd\leq 1$ and $\deg_{x_2}\fd \leq 1$ from Theorem \ref{R3}. Thus there are $a_1,a_2,a_3,a_4,b_1,b_2,b_3,b_4\in\Real$ such that
   \begin{equation*}
\fd(x_1,x_2)=\frac{a_1x_1x_2+a_2x_2+b_1x_1+b_2}{a_3x_1x_2+a_4x_2+b_3x_1+b_4};
 \end{equation*}
 we may assume that $\left( \begin{array}{ccc}
b_1&b_2\\
b_3&b_4
 \end{array} \right)$ is an invertible matrix. The proof is concluded considering the possibilities of the Jordan decomposition of the matrix\\ $\left( \begin{array}{ccc}
b_1&b_2\\
b_3&b_4
 \end{array} \right)^{-1}\left( \begin{array}{ccc}
a_1&a_2\\
a_3&a_4
 \end{array} \right)$ and then defining the functions $g,l_1,l_2\in\Real(x)$ and $h\in\Real(x_1,x_2)$ satisfying i),ii) or iii).

We describe the organization of this paper. In Section 2 we  state  some auxiliary results that will be needed in the next sections. An important part in the proof of Theorem \ref{R3} is to show that if $f(x_1,a)=g(h_a(x_1))$ (resp. $f(x_1,a)=h_a(g(x_1))$) for a family of functions $\{h_a\}_{a\in A}$ in $\field(x)$ with $A$ big enough, then there is $h(x_1,x_2)\in\field(x_1,x_2)$ such that $f(x_1,x_2)=g(h(x_1,x_2))$ (resp. $f(x_1,x_2)=h(g(x_1),x_2)$); this is done using Elimination Theory  in Section 3. Based on some ideas used by Elekes and Ronyai in \cite{ER}, we study the elements $a,a'\in A_2$ (resp. $a,a'\in A_1$) and the functions $g,h_a,h_a'$ such that $f(x_1,a)=g(h_a(x_1))$ and $f(x_1,a')=g(h_{a'}(x_1))$ (resp. $f(a,x_2)=g(h_a(x_2))$ and $f(a',x_2)=g(h_{a'}(x_2))$) when  $\deg g$ is maximal; this is done in Section 4.  The proof of  Theorem \ref{R3} is completed in Section 5. The proofs of the main results are concluded in Section 6.

\section{Preliminaries}
The purpose of this section is to state some results and some direct consequences that will be needed  in the forthcoming sections. We will need some results from Algebra (specially Elimination Theory), Graph Theory and Incidence Theory. In this section $\field$ is a field.

\subsection*{Algebraic preliminaries }

 We start with Lüroth's Theorem.
\begin{thm}
\label{R6}
Let $x$ be an algebraically independent variable over $\field$. For any field $\NF$ such that $\field\subseteq \NF\subseteq \field(x)$, there is $y\in \NF$ such that $\NF=\field(y)$.
\end{thm}
\begin{proof}
See \cite[Thm. 22.19]{Mor}.
\end{proof}
A \emph{monomial order} on $\field[x_1,x_2,\ldots, x_n]$ is a total order $\leq$ on the monomials of $\field[x_1,x_2,\ldots, x_n]$ such that for all $p,q,r\in\field[x_1,x_2,\ldots, x_n]$ monomials,
\begin{enumerate}
\item[i)]$p\leq q$ if and only if $p\cdot r\leq q\cdot r$.
\item[ii)]$p\leq p\cdot q$.
\end{enumerate}
Lexicographic order and graded lex order are example of monomial orders, see \cite[Sec. 2.2]{CLO}.  Fix a monomial order. For any $p\in\field[x_1,x_2,\ldots, x_n]$, write $p$ as a sum of monomials $p=\sum_{j\in J}p_j$. We denote $\LT(p)$ the  \emph{leading term} of $p$ with respect to the monomial order (i.e. $\LT(p)=\max_{j\in J}p_j$), and we denote by $\LC(p)$  the coefficient of the leading term which is known as the \emph{leading coefficient} of $p$. For any ideal $I$ of $\field[x_1,x_2,\ldots, x_n]$, set $\LT(I):=\{\LT(p):\;p\in I\}$.  For any $P$ subset of $\field[x_1,x_2,\ldots, x_n]$, denote by $\left<P\right>$ the ideal of $\field[x_1,x_2,\ldots, x_n]$ generated by $P$; if $P=\{p_1,p_2,\ldots, p_m\}$, write $\left<p_1,p_2,\ldots, p_m\right>:=\left<P\right>$. We say that a finite subset $B=\{p_1,p_2,\ldots, p_m\}$ is a \emph{Gr\"{o}bner basis} of an ideal $I$ (with respect to the fixed monomial order) if
\begin{equation*}
\left<\LT(I)\right>=\left<\LT(p_1),\LT(p_2),\ldots, \LT(p_m)\right>.
\end{equation*}
We say a Gr\"{o}bner basis $B$ of $I$ is \emph{reduced} if 
\begin{enumerate}
\item[i)]$\LC(p)=1$ for all $p\in B$.
\item[ii)]For all $p\in B$, no monomial of $p$ lies in $\left<\LT(B\setminus \{p\})\right>$ (i.e. if $p=\sum_{j\in J}p_j$ with $p_j$ monomials for all $j\in J$, then $\{p_j\}_{j\in J}\cap \left<\LT(B\setminus \{p\})\right>=\emptyset$).
\end{enumerate}  
\begin{prop}
\label{R7}
Let $\leq$ be a monomial order in $\field[x_1,x_2,\ldots, x_n]$ and $I$ be a nonzero ideal of $\field[x_1,x_2,\ldots, x_n]$.
\begin{enumerate}
\item[i)]There exists a unique reduced Gr\"{o}bner basis of $I$ with respect to $\leq$.
\item[ii)]If $B$ is a Gr\"{o}bner basis of $I$ with respect to $\leq$, then it is a basis of $I$.
\end{enumerate} 
\end{prop}
\begin{proof}
Claim i) is \cite[Prop. 2.7.6]{CLO}. Claim ii) follows from \cite[Cor. 2.5.6]{CLO}.
\end{proof}
Next we have the following result of T. Dub\'{e}.
\begin{thm}
\label{R8}
Let $\leq$ be a monomial ordering in $\field[x_1,x_2,\ldots, x_n]$ and $I=\\\left<p_1,p_2,\ldots, p_n\right>$ be an ideal of $\field[x_1,x_2,\ldots, x_n]$ with $d:=\max_{1\leq i\leq n} \deg p_i$. If $B$ is a reduced Gr\"{o}bner basis of $I$ with respect to $\leq$, then
\begin{equation*}
\max_{p\in B}\deg p\leq 2\left(\frac{d^2}{2}+d\right)^{2^{n-1}}.
\end{equation*}
\end{thm}
\begin{proof}
See \cite{Dub}.
\end{proof}
Given an ideal $I$ of $\field[x_1,x_2,\ldots, x_n]$ and $m\in\{1,2,\ldots, n\}$, the $m$-th elimination ideal of $I$ is $I_m:=I\cap \field[x_{m+1},x_{m+2},\ldots, x_n]$. We will need the Elimination Theorem.
\begin{thm}
\label{R9}
Let $I$ be a nonzero ideal of $\field[x_1,x_2,\ldots, x_n]$ and $m\in\{1,2,\ldots, n\}$. For any Gr\"{o}bner basis $B$ of $I$ with respect to the lexicographic order ( i.e. $x_n\leq x_{n-1}\leq \ldots\leq x_1$), we have that $B\cap \field[x_{m+1},x_{m+2},\ldots, x_n]$ is a Gr\"{o}bner basis of  $I_m$ with respect to the lexicographic order.
\end{thm}
\begin{proof}
See \cite[Thm. 3.1.2]{CLO}.
\end{proof}
Let $R$  be a domain and $n\in\Pos$. Denote by $\ori$ the origin in $R^{n+1}$ and set
\begin{equation*}
\Pro^n_R:=R^{n+1}\setminus \{\ori\}/\sim\qquad\text{where }\ab\sim \bb \text{ if there is $r$ a unit of $R$ such that }\ab=r\cdot\bb.
\end{equation*}
  For any homogeneous ideal $I$ in $R[x_1,x_2,\ldots, x_n]$,  set
  \begin{equation*}
  \Zc_R(I):=\{\ab\in\Pro^n_R:\;p(\ab)=0\text{ for all }p\in I\};
  \end{equation*}
  if $I$ is the homogeneous ideal generated by $p_1,p_2,\ldots, p_m$, we  write
  \begin{equation*}
  \Zc_R(p_1,p_2,\ldots, p_m):=\Zc_R(I).
  \end{equation*}
  As usual $\Pro^n_R$ will be considered with the Zariski topology. For any subset $X$ of $\Pro^n_R$, we denote by $\overline{X}$ its closure in $\Pro^n_R$.  The following particular case of \cite[Prop A.7.12]{GP} is a key point in the proof of Theorem \ref{R3}.
  \begin{prop}
  \label{R10}
  Let  $\field$ be an algebraically closed field, $p_0,p_1,\ldots, p_m\in\\\field[x_0,x_1,\ldots, x_n]$ be homogeneous with the same degree and $U$ be a nonempty open subset of $\Pro_\field^n$ such that
\begin{equation*}  
   \Zc_\field(p_0,p_1,\ldots, p_m)\cap U=\emptyset.
\end{equation*}   
    Define the morphism
  \begin{equation*}
  \phi:U\longrightarrow \Pro^m_\field,\qquad \phi(\ab)=[p_0(\ab):p_1(\ab):\ldots:p_m(\ab)],
  \end{equation*}
 and the ideal $I:=\left<p_0-y_0,p_1-y_1,\ldots, p_m-y_m\right>$ in $\field[x_0,\ldots,x_n,y_0,\ldots, y_m]$. Then
  \begin{equation*}
  \overline{\phi(U)}=\Zc_\field\left(I\cap \field[y_0,y_1,\ldots, y_m]\right).
    \end{equation*}
 \end{prop}
 \begin{proof}
 See \cite[Prop A.7.12]{GP}.
 \end{proof}
 Another important tool that we will need is Chevalley's Theorem. We say that a subset of a topological space is \emph{constructible} if it is the finite disjoint union of intersections of open and closed subsets; in particular an irreducible constructible subset is an intersections of an open and a closed subset.
 \begin{thm}
 \label{R11}
 Let  $U$ be a constructible subset of $\Pro^n_\field$ and $\phi:U\rightarrow \Pro^m_\field$ a morphism. Then $\phi(U)$ is a constructible subset of $\Pro^m_\field$.
 \end{thm}
 \begin{proof}
 See \cite[Sec. II.3]{Ha}.
 \end{proof}
 The former statements  will be applied in the following result. 
  \begin{cor}
  \label{R12}
  Let $\field$ be an algebraically closed field, $p_0,p_1,\ldots, p_m\in\\\field[x_0,x_1,\ldots, x_n]$ be homogeneous with the  degree $d$ and $U$ be a nonempty irreducible open subset of $\Pro_\field^n$ such that
\begin{equation*}  
   \Zc_\field(p_0,p_1,\ldots, p_m)\cap U=\emptyset.
\end{equation*}   
    Define the morphism
  \begin{equation*}
  \phi:U\longrightarrow \Pro^m_\field,\qquad \phi(\ab)=[p_0(\ab):p_1(\ab):\ldots:p_m(\ab)].
  \end{equation*}
Then there is a family of polynomials $\Fc=\{r_i\}_{i=1}^{n_r}\cup \{s_i\}_{i=1}^{n_s}$ in $\field[y_0,y_1,\ldots, y_m]$ such that
\begin{equation*}
\max_{1\leq i\leq n_r}\deg r_i\leq 2\left(\frac{d^2}{2}+d\right)^{2^{m+n+1}}.
\end{equation*} 
 and 
  \begin{equation*}
  \phi(U)=\Zc_\field\left(r_1,r_2,\ldots, r_{n_r}\right)\cap \left(\Pro^m_{\field}\setminus \Zc_\field\left(s_1,s_2,\ldots, s_{n_s}\right)\right).
    \end{equation*}
 \end{cor}
 \begin{proof}
 Since $U$ is open, it is constructible.  Thus Theorem \ref{R11} implies that  $\phi(U)$ is constructible. Moreover, inasmuch as $U$ is irreducible, $\phi(U)$ is irreducible so there are $U'$ open in $\Pro^m_\field$ and $C$ closed in $\Pro^m_\field$ such that 
 \begin{equation*}
 \phi(U)=U'\cap C;
 \end{equation*}
 furthermore, since $\overline{\phi(U)}\subseteq C$, we have that
 \begin{equation}
 \label{E1}
 \phi(U)=U'\cap \overline{\phi(U)}.
 \end{equation}
 Now consider the lexicoraphic order $y_m\leq y_{m-1}\leq y_0\leq x_n\leq x_{n-1}\leq \ldots\leq x_0$ in $\field[x_0,x_1,\ldots, x_n,y_0,y_1,\ldots, y_m]$. The ideal   $I:=\left<p_0-y_0,p_1-y_1,\ldots, p_m-y_m\right>$ in $\field[x_0,x_1,\ldots,x_n,y_0,y_1,\ldots, y_m]$ has a reduced Gr\"{o}bner basis $B$ with respect to the lexicographic order by  Proposition \ref{R7} i). Moreover, from Theorem \ref{R8}, 
 \begin{equation}
 \label{E2}
 \max_{p\in B}\deg p\leq 2\left(\frac{d^2}{2}+d\right)^{2^{m+n+1}}.
\end{equation}  
Theorem \ref{R9} implies that $B':=B\cap \field[y_0,y_1,\ldots, y_m]$ is a Gr\"{o}bner basis of $I':=I\cap \field[y_0,y_1,\ldots, y_m]$. Write $B'=\{r_i\}_{i=1}^{n_r}$. On the one hand, $B'$ is a basis of $I'$ by Proposition \ref{R7} ii). On the other hand, Proposition \ref{R10} implies that 
 \begin{equation*}
   \overline{\phi(U)}=\Zc_\field\left(I'\right).
    \end{equation*}
   Thus 
   \begin{equation}
 \label{E3}
  \overline{\phi(U)}=\Zc_\field\left(r_1,r_2,\ldots, r_{n_r}\right).
    \end{equation}
    Since $U'$ is open, the noetherianity of $\field[y_0,y_1,\ldots, y_m]$ implies the existence of a finite family of polynomials $\{s_i\}_{i=1}^{n_s}$ in $\field[y_0,y_1,\ldots, y_m]$ such that
    \begin{equation}
    \label{E4}
    U'=\Pro^m_{\field}\setminus \Zc_\field\left(s_1,s_2,\ldots, s_{n_s}\right).
    \end{equation}
  If we set $\Fc=\{r_i\}_{i=1}^{n_r}\cup \{s_i\}_{i=1}^{n_s}$, then 
  \begin{equation*}
\max_{1\leq i\leq n_r}\deg r_i\leq 2\left(\frac{d^2}{2}+d\right)^{2^{m+n+1}}
\end{equation*} 
 by (\ref{E2}), and
 \begin{equation*}
  \overline{\phi(U)}=\Zc_\field\left(r_1,r_2,\ldots, r_{n_r}\right)\cap \left(\Pro^m_{\field}\setminus \Zc_\field\left(s_1,s_2,\ldots, s_{n_s}\right)\right)
    \end{equation*}
    by (\ref{E1}), (\ref{E3}) and (\ref{E4}).
 \end{proof}
 Now we state a weak version of Bezout's Theorem.
 \begin{thm}
 \label{R13}
 Let $C_1$ and $C_2$ be algebraic curves in $\field^2$ of degree $d_1$ and $d_2$, respectively. Then
 \begin{equation*}
 |C_1\cap C_2|\leq d_1\cdot d_2
 \end{equation*}
 unless $C_1$ and $C_2$ have an irreducible component in common. 
 \end{thm}
 \begin{proof}
 See \cite[Cor. I.7.8]{Ha}.
 \end{proof}
 Let $p,q\in\field[x_1,x_2,\ldots, x_n]$. Write $p(x_1,x_2,\ldots,x_n)=\sum_{i=0}^{n_p}p_{n_p-i}(x_2,\ldots,x_n)x_1^{i}$ and $q(x_1,x_2,\ldots,x_n)=\sum_{i=0}^{n_q}q_{n_q-i}(x_2,\ldots,x_n)x_1^{i}$ with $p_0,p_1,\ldots,p_{n_p}, q_0,\ldots,q_{n_q}\in\field[x_2,x_3,\ldots,x_n]$ and $q_0p_0\neq 0$. Define the $(n_p+n_q)\times (n_p+n_q)$-matrix with coefficients in $\field[x_2,x_3,\ldots,x_n]$.
 \begin{equation*}
\mathrm{Syl}(p,q,x_1)=\begin{array}{cc}
\overbrace{\hspace{100 pt}}^{n_q} & \overbrace{\hspace{100 pt}}^{n_p} \\
 \left( \begin{array}{cccccccccc}
p_0&0&\ldots&0&0\\
p_1&p_0&\ldots&0&0\\
&\ddots&\ddots&&\\
&&&p_0&0\\
&&&p_1&p_0\\
&&&p_2&p_1\\
&&&\ddots&\vdots\\
0&0&\ldots&0&p_{n_p}\\
 \end{array} \right.& \left. \begin{array}{cccccccccc}
q_0&0&\ldots&0&0\\
q_1&q_0&\ldots&0&0\\
&\ddots&\ddots&&\\
&&&q_0&0\\
&&&q_1&q_0\\
&&&q_2&q_1\\
&&&\ddots&\vdots\\
0&0&\ldots&0&q_{n_q}\\
 \end{array} \right)
 \end{array}
  \end{equation*}
  and the \emph{resultant} $\Res(p,q,x_1):=\det(\mathrm{Syl}(p,q,x_1))$. The main property of the resultants that we will need is the following.
  \begin{prop}
  \label{R14}
  Let $p,q\in\field[x_1,x_2,\ldots, x_n]$. Then $\Res(p,q,x_1)=0$ if and only if $p$ and $q$ have a common factor in $\field[x_1,x_2,\ldots, x_n]$ which has positive degree with respect to $x_1$.
  \end{prop}
  \begin{proof}
  See \cite[Prop. 3.6.1]{CLO}.
  \end{proof}
  We apply the previous proposition as follows. 
  \begin{cor}
  \label{R15}
  Let  $p,q\in\field[x_1,x_2]$ be such that $(p,q)=1$ and \\$d:=\max\{\deg p,\deg q\}$. If
    \begin{align*}
  A:=\{a\in\field:\;(p(x_1,a),q(x_1,a))\neq 1\},
  \end{align*}
  then
  \begin{equation*}
  |A|\leq d^{2d}.
  \end{equation*}
  \end{cor}
  \begin{proof}
  Let  $p_0,p_1,\ldots,p_{n_p}, q_0,\ldots,q_{n_q}$ be the elements of $\field[x_2]$ such that $p(x_1,x_2)=\sum_{i=0}^{n_p}p_{n_p-i}(x_2)x_1^{i}$, $q(x_1,x_2)=\sum_{i=0}^{n_q}q_{n_q-i}(x_2)x_1^{i}$ and   $q_0p_0\neq 0$. From Proposition \ref{R14}, we have that $\Res(p,q,x_1)\neq 0$ since $(p,q)=1$. However, $\mathrm{Syl}(p,q,x_1)$ is a $(n_p+n_q)\times (n_p+n_q)$-matrix with entries $p_0,p_1,\ldots,p_{n_p}, q_0,\ldots,q_{n_q},0$; each polynomial $p_i$ or $q_i$ has degree at most $d$ so $\Res(p,q,x_1)$ is a polynomial in $\field[x_2]$ with degree at most  $d^{n_p+n_q}$. Moreover, since $\max\{n_p,n_q\}\leq d$,  $\Res(p,q,x_1)$ is a polynomial in $\field[x_2]$ with degree at most  $d^{2d}$; thus if $B$ is the set of roots of $\Res(p,q,x_1)$ in $\field$, then
  \begin{align}
  \label{E5}
   |B|\leq d^{2d}.
  \end{align}
   Take $a\in \field$ such that $(p(x_1,a),q(x_1,a))\neq 1$. Hence Proposition \ref{R14} implies that $\Res(p(x_1,a),q(x_1,a),x_1)=0$, and so $a$ is a root of $\Res(p(x_1,x_2),q(x_1,x_2),x_1)\in\field[x_2]$. This shows that $A\subseteq B$, and (\ref{E5}) completes the proof of our claim.
  \end{proof}
  Let $a,b,c,d\in\field$ and $p(x_1,x_2):=ax_1x_2+bx_1+cx_2+d\in\field[x_1,x_2]$. Since $\deg_{x_1}p,\deg_{x_2}p\leq 1$,  if $p$ is reducible, then there exist $a',b',c',d'\in\field$ such that $p(x_1,x_2)=(a'x_1+b')(c'x_2+d')$; in particular $ad-bc=0$.  Thus we have the following fact.
  \begin{rem}
  \label{R16}
   Let $a,b,c,d\in\field$ be such that $ad-bc\neq 0$. Then the polynomial $ax_1x_2+bx_1+cx_2+d\in\field[x_1,x_2]$ is irreducible.
  \end{rem}
  \begin{lem}
  \label{R17}
  Let $p_1,q_1\in\field[x_1]$ and $p_2,q_2\in\field[x_2]$ be such that $(p_1,q_1)=1$ and $(p_2,q_2)=1$. If $\max\{\deg p_1,\deg q_1\}=\max\{\deg p_2,\deg q_2\}=1$, then the polynomial $p_1(x_1)q_2(x_2)-q_1(x_1)p_2(x_2)\in\field[x_1,x_2]$ is irreducible.
  \end{lem}
  \begin{proof}
  Since  $\max\{\deg p_1,\deg q_1\}=\max\{\deg p_2,\deg q_2\}=1$, there are $a_i,b_i,c_i,d_i\in\field$ for $i\in\{1,2\}$ such that 
  \begin{align*}
  p_1(x)&=a_1x_1+b_1&p_2(x_2)&=a_2x_2+b_2\\
  q_1(x)&=c_1x_1+d_1&q_2(x_2)&=c_2x_2+d_2.
  \end{align*}
  Inasmuch as $(p_1,q_1)=1$ and $(p_2,q_2)=1$, we have that
  \begin{align}
  \label{E6}
  a_1d_1-b_1c_1&\neq 0&a_2d_2-b_2c_2\neq 0.
  \end{align}
  Set $r(x_1,x_2):=p_1(x_1)q_2(x_2)-q_1(x_1)p_2(x_2)$. Then
  \begin{align}
  \label{E7}
  r&=a_1c_2x_1x_2+a_1d_2x_1+b_1c_2x_2+b_1d_2-(c_1a_2x_1x_2+d_1a_2x_2+c_1b_2x_1+d_1b_2)\nonumber\\
  &=(a_1c_2-c_1a_2)x_1x_2+(a_1d_2-c_1b_2)x_1+(b_1c_2-d_1a_2)x_2+(b_1d_2-d_1b_2).
  \end{align}
  We have that
  \begin{align}
  \label{E8}
  (a_1c_2-c_1a_2)(b_1d_2-d_1b_2)-(a_1d_2-c_1b_2)(b_1c_2-d_1a_2)&=\nonumber\\
  a_1a_2d_1d_2+b_1b_2c_1c_2-a_1b_2d_1c_2-b_1a_2c_1d_2&=\nonumber\\
  (a_1d_1-b_1c_1)(a_2d_2-b_2c_2).&
  \end{align}
  Thus, from (\ref{E6}) and (\ref{E8}), we have that 
  \begin{equation}
  \label{E9}
  (a_1c_2-c_1a_2)(b_1d_2-d_1b_2)-(a_1d_2-d_1a_2)(b_1c_2-c_1b_2)\neq 0
  \end{equation}
  From Remark \ref{R16}, (\ref{E7}) and (\ref{E9}), we conclude that $r(x_1,x_2)$ is irreducible. 
  \end{proof}
  We say that $f,g\in\field(x)$ are \emph{equivalent} if there is $h\in\field(x)$ such that $f=g\circ h$ and $\deg h=1$. For any decompositions $f=g_1\circ h_1$ and $f=g_2\circ h_2$ of $f\in\field(x)$, we say that they are \emph{(linearly) equivalent} if $g_1$ and $g_2$ are equivalent. Elekes and R\'{o}nyai proved an useful result.
  \begin{prop}
  \label{R18}
  Let $f\in\field(x)$ and $d:=\deg f$. Then $f$ cannot have more than $2^d$ nonequivalent decompositions.
  \end{prop}
  \begin{proof}
  See \cite[Prop. 9]{ER}.
  \end{proof}
 \subsection*{Graph Theory preliminaries}
 We need two Graph Theory results. We say that a graph $G=(V,E)$ is \emph{simple} if $G$ has neither loops nor parallel edges. For any $v\in V$, we denote by $d_G(v)$ its degree in $G$. Given a colouring $E=\bigcup_{i\in I} E_i$ of the edges of $G$, we say that the subgraph $G'=(V',E')$ of $G$ is \emph{monochromatic} if there is $i\in I$ such that $E'\subseteq E_i$.
 \begin{lem}
 \label{R19}
 For every $c>0$ and $n\in\Pos$, there is $c_2=c_{2,c,n}>0$ with the following property. Let $G=(V,E)$ be a finite simple graph such that $d_G(v)\geq c|V|$ for all $v\in V$. Take a colouring $E=\bigcup_{i\in I}E_i$ such that at most $n$ colors meet in each vertex of $G$. Then there is a monochromatic subgraph $G'=(V',E')$ of $G$ such that $|E'|\geq c_2|E|$.
 \end{lem} 
 \begin{proof}
 See \cite[Cor. 17]{ER}.
 \end{proof}
 
 \begin{lem}
 \label{R20}
 Let $G=(V,E)$ be a finite simple graph  and $d\in\Nat$ such that $\max_{v\in V}d_G(v)=d$. Then there exists a partition $V=\biguplus_{i=1}^{d+1} V_i$ of $V$ such that for all $i\in\{1,2,\ldots, d+1\}$,  $V_i\neq \emptyset$ and $\{v,v'\}\not\in E$ for all $v,v'\in V_i$. 
 \end{lem}
 \begin{proof}
 We prove the lemma by induction on $|V|$. Since $\max_{v\in V}d_G(v)=d$, we have that $|V|\geq d+1$. If $|V|=d+1$, then write $V=\{v_1,v_2,\ldots, v_{d+1}\}$; taking $V_i=\{v_i\}$ in this case, we get the conclusion of the lemma. Assume that the statement holds for all graph $G=(V,E)$ with $|V|<k$ and we prove it when $|V|=k$. Furthermore, we assume that $|V|>d+1$ from now on. Let $v_0\in V$ be such that $d_G(v_0)=d$ and $v_1,v_2,\ldots, v_d\in V$ be the neighbours of $v_0$. Since $|V|>d+1$, there is $v\in V\setminus\{v_0,v_1,v_2,\ldots, v_d\}$, and set $V':=V\setminus\{v\}, E':=\left\{\{v',v''\}\in E:\;v',v''\in V'\right\}$ and $G'=(V',E')$. Since $v_0\in V'$, 
 \begin{equation*}
 d=d_G(v_0)\leq \max_{v'\in V'}d_{G'}(v')\leq \max_{v'\in V}d_G(v')=d.
 \end{equation*}
 Thus we can apply the induction to $G'$, and therefore there is a partition $V'=\biguplus_{i=1}^{d+1} V'_i$  such that for all $i\in\{1,2,\ldots, d+1\}$,  $V'_i\neq \emptyset$ and $\{v',v''\}\not\in E$ for all $v',v''\in V'_i$. Since $d_G(v)\leq d<d+1$, there is $i\in\{1,2,\ldots,d+1\}$ such that $\{v,v'\}\not\in E$ for all $v'\in V'_i$; assume without loss of generality that $i=1$. Set $V_1:=V'_1\cup\{v\}$ and $V'_i=V_i$ for all $i\in\{2,3,\ldots, d+1\}$. On the one hand, $V=\biguplus_{i=1}^{d+1} V_i$ is a partition with nonempty subsets $V_1,V_2,\ldots, V_{d+1}$ inasmuch as $V'=\biguplus_{i=1}^{d+1} V'_i$ is a partition with nonempty subsets $V'_1,V'_2,\ldots, V'_{d+1}$. On the other hand, the construction of the sets $V_1,V_2,\ldots, V_{d+1}$ implies that for all $i\in \{1,2,\ldots, d+1\}$, we have that $\{v',v''\}\not\in E$ for all $v',v''\in V_i$.  
 \end{proof}
 \subsection*{Incidence Theory preliminaries}
 The Szemer\'{e}di-Trotter type results have been fundamental parts in the proofs of Elekes-R\'{o}nyai type results. For a subset $A$ of $\Com^2$ and a family of algebraic curves  $\Cc$ in $\Com^2$, set
 \begin{equation*}
 I(A,\Cc)=\{(\ab,C)\in P\times \Cc:\;\ab\in C \}.
\end{equation*}   
In this paper we will need the following result proven by Solymosi and de Zeeuw.
 \begin{thm}
 \label{R21}
 Let $A_1$ and $A_2$ be nonempty finite subsets of $\Com$ with $|A_1|=|A_2|$, $d,m\in\Nat$ and $\Cc$ be a finite family of algebraic curves  in $\Com^2$ of degree at most $d$ such that no two of them have a common component. Set $A:=A_1\times A_2$, and let $I$ be a subset of $I(A,\Cc)$. Assume that for all $\ab,\ab'\in A$,
 \begin{equation*}
 \left|\left\{C\in \Cc:\;\{(\ab,C),(\ab',C)\}\subseteq I\right\}\right|< m.
 \end{equation*}
 Then
 \begin{equation*}
 |I|=O_{d,m}\left(|A|^{\frac{2}{3}}|\Cc|^{\frac{2}{3}}+|A|+|\Cc|\right).
 \end{equation*}
 \end{thm}
\begin{proof}
See \cite[Cor. 16]{SZ}.
\end{proof} 
Let $f\in\Com(x_1,x_2)$. For any $\ab_1=(a_1,a'_1),\ab_2=(a_2,a'_2)\in\Com^2$, write 
\begin{align*}
f(x_1,a_2)&=\frac{p_{a_2}(x_1)}{q_{a_2}(x_1)}&f(a_1,x_2)=\frac{r_{a_1}(x_2)}{s_{a_1}(x_2)}\\
f(x_1,a'_2)&=\frac{p_{a'_2}(x_1)}{q_{a'_2}(x_1)}&f(a'_1,x_2)=\frac{r_{a'_1}(x_2)}{s_{a'_1}(x_2)}.
\end{align*}
for some polynomials $p_{a_2},q_{a_2},p_{a'_2},q_{a'_2}\in \Com[x_1]$ and $r_{a_1},s_{a_1},r_{a'_1},s_{a'_1}\in \Com[x_2]$ with $(p_{a_2},q_{a_2})=(p_{a'_2},q_{a'_2})=1$ and $(r_{a_1},s_{a_1})=(r_{a'_1},s_{a'_1})=1$. We define the complex algebraic curves
\begin{align*}
C^{(1)}_{f,\ab_2}&:=\left\{(b_1,b_2)\in\Com^2:\;p_{a_2}(b_1)q_{a'_2}(b_2)-p_{a'_2}(b_2)q_{a_2}(b_1)=0\right\}\\
C^{(2)}_{f,\ab_1}&:=\left\{(b_1,b_2)\in\Com^2:\;r_{a_1}(b_1)s_{a'_1}(b_2)-r_{a'_1}(b_2)s_{a_1}(b_1)=0\right\}.
\end{align*}
 For any $A$ subset of $\Com^2$, write
 \begin{align*}
 \Cc^{(1)}_{f,A}&:=\left\{C^{(1)}_{f,\ab}:\;\ab \in A\right\}\\
 \Cc^{(2)}_{f,A}&:=\left\{C^{(2)}_{f,\ab}:\;\ab \in A\right\}.
 \end{align*}
Assume that $(a_1,a_2),(a'_1,a'_2)\in \Dom f$. Then $\ab_1\in C^{(1)}_{f,\ab_2}$ if and only if $f(a_1,a_2)=f(a'_1,a'_2)$; in the same way,  $\ab_2\in C^{(2)}_{f,\ab_1}$ if and only if $f(a_1,a_2)=f(a'_1,a'_2)$. As a consequence of these facts, we get the following duality.
 \begin{rem}
 \label{R22}
 For any $f\in\Com(x_1,x_2)$ and  $\ab_1=(a_1,a'_1),\ab_2=(a_2,a'_2)\in \Com^2$ such that $(a_1,a_2),(a'_1,a'_2)\in \Dom f$, we have that $\ab_1\in C^{(1)}_{f,\ab_2}$  if and only if $\ab_2\in C^{(2)}_{f,\ab_1}$. 
 \end{rem}
 In the proof of Theorem \ref{R3}, we wont be able to apply Theorem \ref{R21} directly. We will need the following consequence of Theorem \ref{R21}.
 
\begin{thm}
 \label{R23}
 Let  $d,m\in\Nat$, $f\in\Com(x_1,x_2)$ with $\deg f\leq d$, and  $A_1,A_2$ be nonempty finite subsets of $\Com$ with $|A_1|=|A_2|$ and $A_1\times A_2\subseteq \Dom f$. Take $B_1\subseteq A_1\times A_1$ and $B_2\subseteq A_2\times A_2$ such that $|(A_1\times A_1)\setminus B_1|\leq m|A_1|$ and $|(A_2\times A_2)\setminus B_2|\leq m|A_2|$. Assume that for any irreducible curve $C$ in $\Com^2$,
 \begin{equation}
 \label{E10}
 \left|\left\{\ab\in B_2:\;C^{(1)}_{f,\ab}\supseteq  C\right\}\right |, \left|\left\{\ab\in B_1:\;C^{(2)}_{f,\ab}\supseteq  C\right\}\right|< m
 \end{equation}
 Then
 \begin{align*}
I\left(A_1\times A_1,\Cc^{(1)}_{f,A_2\times A_2}\right)&=O_{m,d}\left(|A_1|^{\frac{8}{3}}\right)\\
I\left(A_2\times A_2,\Cc^{(2)}_{f,A_1\times A_1}\right)&=O_{m,d}\left(|A_1|^{\frac{8}{3}}\right).
 \end{align*}
 \end{thm}
 \begin{proof}
 Define the graph $G=(V,E)$ with $V:=B_2$ and $E$ the set of pairs $\{\ab,\ab'\}$ with $\ab,\ab'\in B_2$ such that $C^{(1)}_{f,\ab}$ and $C^{(1)}_{f,\ab'}$ have an irreducible component in common. For all $\ab\in\Com^2$, since $\deg f\leq d$, the algebraic curve $C^{(1)}_{f,\ab}$ has degree at most $2d$; in particular $C^{(1)}_{f,\ab}$ has at most $2d$ irreducible components. Now note that each irreducible curve in $\Com^2$ can be shared by less than $m$ curves in $\Cc^{(1)}_{f,B_2}$  by (\ref{E10}). Thereby the degree of each vertex of $G$ can be at most $2md$; set $d_2:=\max_{v\in V}d_G(v)\leq 2md$. From Lemma \ref{R20}, there is a partition $B_{2}=\biguplus_{i=1}^{d_2+1} B_{2,i}$ such that the subsets of the partition are not empty and $\{\ab,\ab'\}\not\in E$ for all $\ab,\ab'\in B_{2,i}$ with $\ab\neq \ab'$ and $i\in\{1,2,\ldots,d_2+1\}$. This means that for all  $i\in\{1,2,\ldots,d_2+1\}$  and $\ab,\ab'\in B_{2,i}$, the curves   $C^{(1)}_{f,\ab}$ and $C^{(1)}_{f,\ab'}$ do not  have an irreducible component in common; in particular Theorem \ref{R13} yields that
 \begin{equation}
 \label{E11}
 \left|C^{(1)}_{f,\ab}\cap C^{(1)}_{f,\ab'}\right|\leq \deg C^{(1)}_{f,\ab}\cdot \deg C^{(1)}_{f,\ab'}\leq 4d^2.
 \end{equation}
 Proceeding in the same way, there is $d_1\leq 2md$ and $B_{1}=\biguplus_{i=1}^{d_1+1} B_{1,i}$ such that the subsets of the partition are not empty and  for all  $i\in\{1,2,\ldots,d_1+1\}$  and $\ab,\ab'\in B_{1,i}$ with $\ab\neq \ab'$, the curves   $C^{(2)}_{f,\ab}$ and $C^{(2)}_{f,\ab'}$ do not share an irreducible component. Let $i\in\{1,2,\ldots, d_1\}$ and $j\in\{1,2,\ldots, d_2\}$. For any two different curves in $\Cc^{(2)}_{f,B_{1,i}}$, they  do not share an irreducible component. For any $\ab,\ab'\in B_{2,j}$, Remark \ref{R22} and (\ref{E11}) imply that there are at most $4d^2$ elements $C$ of $\Cc^{(2)}_{f,A_1\times A_1}$ such that $\ab,\ab'\in C$. The last two claims make possible to apply Theorem \ref{R21} to the set $A_2\times A_2\subseteq \Com^2$, the family of curves $\Cc^{(2)}_{f,A_1\times A_1}$  and the incidence subset
 \begin{equation*}
 I\left(B_{2,j},\Cc^{(2)}_{f,B_{1,i}}\right)\subseteq I\left(A_2\times A_2,\Cc^{(1)}_{f,A_1\times A_1}\right).
 \end{equation*}
 Therefore Theorem \ref{R21} leads to
 \begin{align}
 \label{E12}
I\left(B_{2,j},\Cc^{(2)}_{f,B_{1,i}}\right)&=O_{d,m}\left(|A_2\times A_2|^{\frac{2}{3}}\left|\Cc^{(2)}_{f,A_1\times A_1}\right|^{\frac{2}{3}}+|A_2\times A_2|+\left|\Cc^{(2)}_{f,A_1\times A_1}\right|\right)\nonumber\\
&=O_{d,m}\left(|A_1|^{\frac{8}{3}}\right).
 \end{align}
 Inasmuch as $d_1,d_2\leq 2dm$, (\ref{E12}) yields that
 \begin{equation}
 \label{E13}
 I\left(B_{2},\Cc^{(2)}_{f,B_{1}}\right)=\sum_{j=1}^{d_2}\sum_{i=1}^{d_1}I\left(B_{2,j},\Cc^{(2)}_{f,B_{1,i}}\right)=O_{d,m}\left(|A_1|^{\frac{8}{3}}\right).
 \end{equation}
 Now note that for any $\ab\in A_2\times A_2$, the curve $C^{(1)}_{f,\ab}$ has degree at most $2d$; thus for any $a_1\in A_1$, there are at most $2d$ elements $a_1'\in A_1$ such that $(a_1,a'_1)\in C^{(1)}_{f,\ab}$; thus, from Remark \ref{R22}, this means that given $a_1\in A_1$, there are at most $2d$ elements $a_1'\in A_1$ such that $\ab\in C^{(2)}_{f,(a_1,a'_1)}$. This leads to the inequality
 \begin{equation*}
  \left|\left\{(a_1,a'_1)\in A_1\times A_1:\;\ab\in C^{(2)}_{f,(a_1,a'_1)}\right\}\right|\leq 2d|A_1|,
 \end{equation*}
 and then 
 \begin{equation}
 \label{E14}
 I\left(\{\ab\},\Cc^{(2)}_{f,A_1\times A_1}\right)\leq 2d|A_1|.
 \end{equation}
 Thus (\ref{E14}) yields that
 \begin{align*}
 I\left((A_2\times A_2)\setminus B_2,\Cc^{(2)}_{f,A_1\times A_1}\right)&=\sum_{\ab\in(A_2\times A_2)\setminus B_2} I\left(\{\ab\},\Cc^{(2)}_{f,A_1\times A_1}\right)\nonumber\\
 &\leq 2d|(A_2\times A_2)\setminus B_2||A_1|
 \end{align*}
  and the inequality $|(A_2\times A_2)\setminus B_2|\leq m|A_2|$ implies that
  \begin{equation}
   \label{E15}
    I\left((A_2\times A_2)\setminus B_2,\Cc^{(2)}_{f,A_1\times A_1}\right)\leq 2md|A_1||A_2|=2md|A_1|^2.
  \end{equation}
  In the same way it is proven that
  \begin{equation}
  \label{E16}
  I\left((A_1\times A_1)\setminus B_1,\Cc^{(1)}_{f,A_2\times A_2}\right)\leq 2md|A_1|^2.
  \end{equation}
 Remark \ref{R22} and (\ref{E16}) lead to
  \begin{equation*}
    I\left(A_2\times A_2, \Cc^{(2)}_{f,(A_1\times A_1)\setminus B_1}\right)\leq 2md|A_1|^2,
  \end{equation*}
  and hence, since $B_2\subseteq A_2\times A_2$, we conclude from the previous inequality that 
   \begin{equation}
   \label{E17}
    I\left(B_2, \Cc^{(2)}_{f,(A_1\times A_1)\setminus B_1}\right)\leq 2md|A_1|^2.
  \end{equation}
  From (\ref{E13}), (\ref{E15}) and (\ref{E17}), 
  \begin{align*}
  I\left(A_2\times A_2,\Cc^{(2)}_{f,A_1\times A_1}\right)=&I\left(B_{2},\Cc^{(2)}_{f,B_{1}}\right)+ I\left((A_2\times A_2)\setminus B_2,\Cc^{(2)}_{f,A_1\times A_1}\right)\\
  &+I\left(B_2, \Cc^{(2)}_{f,(A_1\times A_1)\setminus B_1}\right)\\
  =&O_{m,d}\left(|A_1|^{\frac{8}{3}}\right).
  \end{align*}
  Proceeding symmetrically it is proven that
  \begin{equation*}
  I\left(A_1\times A_1,\Cc^{(1)}_{f,A_2\times A_2}\right)
  =O_{m,d}\left(|A_1|^{\frac{8}{3}}\right).\qedhere
\end{equation*}   
\end{proof}
\section{Lifting rational functions}
Let $\field\in\{\Real,\Com\}$. We will denote by $\Pc_\field$ the family of rings of polynomials with coefficients in $\field$ (we consider $\field\in\Pc_\field$); thus  if $R=\field[x_1,x_2,\ldots,x_n]$ and $x$ is an algebraically independent variable over $R$, then $R[x]\in \Pc_\field$. We start this section with two lemmas that characterize the functions that can be be factorized through a given function. 

For $n_1,n_2,m_1,m_2,j\in\Nat$, write
\begin{align*}
[n_1]&:=\{0,1,\ldots, n_1\}\\
[n_1]_j^{m_1}&:=\left\{\jb=(j_1,j_2,\ldots, j_{m_1})\in [n_1]^{m_1}:\;\sum_{k=1}^{m_1}j_k=j\right\}\\
([n_1]^{m_1}\times [n_2]^{m_2})_j&:=\left\{\jb=(j_1,\ldots, j_{m_1+m_2})\in [n_1]^{m_1}\times [n_2]^{m_2}:\;\sum_{k=1}^{m_1+m_2}j_k=j\right\}.
\end{align*}
To abbreviate the notation, for each $\jb\in [n_1]_j^{m_1}$ (resp. $\jb\in ([n_1]^{m_1}\times [n_2]^{m_2})_j$), its entries will be  $\jb=(j_1,j_2,\ldots, j_{m_1})$ (resp. $\jb=(j_1,j_2,\ldots, j_{m_1+m_2})$). Let $g\in\field(x)$ and $p,q\in\field[x]$ be such that $(p,q)=1$ and $g=\frac{p}{q}$. Write $p(x)=\sum_{i=0}^{m_p}b_{p,i}x^i$ and  $q(x)=\sum_{i=0}^{m_q}b_{q,i}x^i$ where $m_p:=\deg p$ and $m_q:=\deg q$; set $m:=\max\{m_p,m_q\}$. For any $n\in\Nat$, define the families of polynomials $\Fc^{(1)}_{g,n}:=\{p_{1,i}\}_{i=0}^{nm}\cup \{q_{1,i}\}_{i=0}^{nm}$ and $\Fc^{(2)}_{g,n}:=\{p_{2,i}\}_{i=0}^{nm}\cup \{q_{2,i}\}_{i=0}^{nm}$ in $\field[z_0,z_1,\ldots,z_n,w_0,w_1,\ldots,w_n]$ for all $j\in \{0,1,\ldots, nm\}$ as follows:
\begin{align*}
p_{1,j}&:=\sum_{i=0}^{m_p}b_{p,i}\sum_{\jb\in[n]^{m}_j}z_{j_1}z_{j_2}\ldots z_{j_i}w_{j_{i+1}}\ldots w_{j_m}\\
q_{1,j}&:=\sum_{i=0}^{m_q}b_{q,i}\sum_{\jb\in[n]^{m}_j}z_{j_1}z_{j_2}\ldots z_{j_i}w_{j_{i+1}}\ldots w_{j_m}\\
p_{2,j}&:=\sum_{i=0}^n\left(\sum_{\jb\in([m_p]^i\times [m_q]^{n-i})_j}b_{p,j_1}b_{p,j_2}\ldots b_{p,j_i}b_{q,j_{i+1}}\ldots b_{q,j_n}\right)z_i\\
q_{2,j}&:=\sum_{i=0}^n\left(\sum_{\jb\in([m_p]^i\times [m_q]^{n-i})_j}b_{p,j_1}b_{p,j_2}\ldots b_{p,j_i}b_{q,j_{i+1}}\ldots b_{q,j_n}\right)w_i.
\end{align*}
For $R\in\Pc_\field$, set
\begin{align*}
V_{R,n}&:=\left\{(a_{p,0},\ldots,a_{p,n},a_{q,0},\ldots, a_{q,n})\in R^{2n+2}:\;(a_{q,0},\ldots,a_{q,n})\neq (0,\ldots, 0)\right\}.
\end{align*}
Note that the definition of  the families $\Fc^{(i)}_{g,n}$ for $i\in\{1,2\}$ depends on the pair of  polynomials $\{p,q\}$ chosen to represent the quotient $g=\frac{p}{q}$; however, for any two  pair of polynomials $\{p_1,q_1\}$ and $\{p_2,q_2\}$ as above, there is $a\in\field$ such that $p_1=a\cdot p_2$ and $p_2=a\cdot q_2$; therefore the families $\Fc^{(i)}_{g,n}$ for $i\in\{1,2\}$ defined by both pairs are identical up to multiplication by scalars. We also have the following fact about the defined families.
\begin{rem}
\label{R24}
For any $g\in \field(x)$ with $\deg g=m$ and $n\in\Nat$, $\Fc^{(1)}_{g,n}$ is a family of homogeneous polynomials of degree $m$, while $\Fc^{(2)}_{g,n}$ is a family of homogeneous polynomials of degree $1$.
\end{rem}
 \begin{lem}
 \label{R25}
 Let $\field\in\{\Real,\Com\}$, $R\in\Pc_\field$, $x$ be an algebraically independent variable over $R$, $g\in \field(x)$ with $\deg g=m$, $f\in R(x)$ and $n\in\Nat$. 
 \begin{enumerate}
 \item[i)]There is $h\in R(x)$ with $\deg_x h\leq n$ such that $f=g\circ h$ if and only if there is $\ab\in V_{R,n}$ such that 
 \begin{equation}
 \label{E18}
 f(x)=\frac{\sum_{i=0}^{nm}p_{1,i}(\ab)x^i}{\sum_{i=0}^{nm}q_{1,i}(\ab)x^i}
 \end{equation}
 where  $\{p_{1,i}\}_{i=0}^{nm}\cup \{q_{1,i}\}_{i=0}^{nm}=\Fc^{(1)}_{g,n}$. 
 \item[ii)]There is $h\in R(x)$ with $\deg_x h\leq n$ such that $f=h\circ g$ if and only if there is $\ab\in V_{R,n}$ such that 
 \begin{equation*}
 f(x)=\frac{\sum_{i=0}^{nm}p_{2,i}(\ab)x^i}{\sum_{i=0}^{nm}q_{2,i}(\ab)x^i}
 \end{equation*}
 where  $\{p_{2,i}\}_{i=0}^{nm}\cup \{q_{2,i}\}_{i=0}^{nm}=\Fc^{(2)}_{g,n}$. 
 \end{enumerate}
 \end{lem}
 \begin{proof}
 Write $g=\frac{p}{q}$ with $p,q\in\field[x], (p,q)=1$, and also $m_p:=\deg p$ and $m_q:=\deg q$; assume without loss of generality that $m=m_p$ (the case $m=m_q$ is done symmetrically). Set $p(x)=\sum_{i=0}^{m_p}b_{p,i}x^i$ and $q(x)=\sum_{i=0}^{m_q}b_{q,i}x^i$. For any $\ab=(a_{p,0},a_{p,1},\ldots,a_{p,n},a_{q,0},\ldots, a_{q,n})\in V_{R,n}$, define $p_\ab(x):=\sum_{i=0}^na_{p,i}x^i$ and $q_\ab(x):=\sum_{i=0}^na_{q,i}x^i$. Then  the definition of $\Fc^{(1)}_{g,n}$ leads to
 \begin{align}
 \label{E19}
 \sum_{i=0}^{m_p}b_{p,i}p_\ab(x)^iq_\ab(x)^{m_p-i}&=\sum_{j=0}^{nm}p_{1,j}(\ab)x^j\nonumber\\
 \left(\sum_{i=0}^{m_q}b_{q,i}p_\ab(x)^iq_\ab(x)^{m_q-i}\right)q_\ab(x)^{m_p-m_q}&=\sum_{j=0}^{nm}q_{1,j}(\ab)x^j,
 \end{align}
 and the definition of $\Fc^{(2)}_{g,n}$ yields that
 \begin{align}
 \label{E20}
 \sum_{i=0}^{n}a_{p,i}p(x)^iq(x)^{n-i}&=\sum_{j=0}^{nm}p_{2,j}(\ab)x^j\nonumber\\
 \sum_{i=0}^{n}a_{q,i}p(x)^iq(x)^{n-i}&=\sum_{j=0}^{nm}q_{2,j}(\ab)x^j. \end{align}
 To show i), first we assume  that there is $h(x)=\frac{\sum_{i=0}^na_{p,i}x^i}{\sum_{i=0}^na_{q,i}x^i}\in R(x)$ such that $\deg h\leq n$ and $f=g\circ h$. Set $\ab=(a_{p,0},a_{p,1},\ldots,a_{p,n},a_{q,0},\ldots, a_{q,n})$ so that $h=\frac{p_\ab}{q_\ab}$. Thus
 \begin{align}
 \label{E21}
 f(x)=g(h(x))=\frac{\sum_{i=0}^{m_p}b_{p,i}\left(\frac{p_\ab}{q_\ab}\right)^i}{\sum_{i=0}^{m_q}b_{q,i}\left(\frac{p_\ab}{q_\ab}\right)^i}=\frac{\sum_{i=0}^{m_p}b_{p,i}p_\ab^iq_\ab^{m_p-i}}{\left(\sum_{i=0}^{m_q}b_{q,i}p_\ab^iq_\ab^{m_q-i}\right)q_\ab^{m_p-m_q}}.
 \end{align}
 Then the equalities in (\ref{E19}) and (\ref{E21}) lead to (\ref{E18}). Now assume (\ref{E18}). Define  $h:=\frac{p_\ab}{q_\ab}$. On the one hand, $\deg_x h\leq \max\{\deg_x p_\ab,\deg_x q_\ab\}\leq n$. On the other hand,
 \begin{align*}
 f(x)&=\frac{\sum_{i=0}^{nm}p_{1,i}(\ab)x^i}{\sum_{i=0}^{nm}q_{1,i}(\ab)x^i}&\Big(\text{by (\ref{E18})}\Big)\\
 &=\frac{\sum_{i=0}^{m_p}b_{p,i}p_\ab^iq_\ab^{m_p-i}}{\left(\sum_{i=0}^{m_q}b_{q,i}p_\ab^iq_\ab^{m_q-i}\right)q_\ab^{m_p-m_q}}&\Big(\text{by (\ref{E19})}\Big)\\
 &=\frac{\sum_{i=0}^{m_p}b_{p,i}\left(\frac{p_\ab}{q_\ab}\right)^i}{\sum_{i=0}^{m_q}b_{q,i}\left(\frac{p_\ab}{q_\ab}\right)^i}\\
 &=g(h(x)),
 \end{align*}
  and this completes the proof of i). The proof of ii) is done analogously using (\ref{E20}) instead of (\ref{E19}).
 \end{proof}
 \begin{rem}
 \label{R26}A consequence of Lemma \ref{R25} is that for all $\ab\in V_{R,n}$ and $i\in\{1,2\}$, 
 \begin{equation*}
 (q_{i,0}(\ab),q_{i,1}(\ab),\ldots,q_{i,nm}(\ab))\neq (0,0,\ldots,0).
\end{equation*}  
 \end{rem}
 For $\field\in\{\Real,\Com\}$, $R\in\Pc_\field$, $m,n\in\Nat$, $i\in\{1,2\}$ and $g\in\field(x)$ with $\deg g=m$, set
 \begin{equation*}
 U_{R,n}:=\left\{[a_{p,0}:\ldots:a_{p,n}:a_{q,0}:\ldots:a_{q,n}]\in \Pro_R^{2n+1}:\;(a_{q,0},\ldots,a_{q,n})\neq (0,\ldots, 0)\right\}.
 \end{equation*}
  and the morphisms
  \begin{align*}
  \phi^{(i)}_{g,R,n}:U_{R,n}\longrightarrow \Pro_R^{2nm+1},&\\
  \phi^{(i)}_{g,R,n}(\ab)=[p_{i,0}(\ab):p_{i,1}(\ab):\ldots:p_{i,nm}(\ab):q_{i,0}(\ab):q_{i,1}(\ab):\ldots:q_{i,nm}(\ab)]&
  \end{align*}
  where  $\{p_{i,j}\}_{j=0}^{nm}\cup \{q_{i,j}\}_{j=0}^{nm}=\Fc^{(i)}_{g,n}$. 
  \begin{thm}
  \label{R27}
   Let $\field\in\{\Real,\Com\}$,  $n,m\in\Pos$ and $g\in\field(x)$ be such that $\deg g=m$.
   \begin{enumerate}
   \item[i)]There is a family of polynomials  $\Gc^{(1)}_{g,n}=\{r_{1,j}\}_{j=1}^{n_1}\cup \{s_{1,j}\}_{j=1}^{m_1}$ in \\$\Com[z_0,z_1,\ldots, z_{nm},w_0,w_1,\ldots,w_{nm}]$. With the following two properties.
   \begin{enumerate}
   \item[1)]\begin{equation*}
   \max_{1\leq i\leq n_1}\deg r_{1,i}\leq 2\left(\frac{n^2m^2}{2}+nm\right)^{2^{4nm+3}}.
   \end{equation*}
   \item[2)]For all $R\in \Pc_\field$, $x$ algebraically independent variable over $R$ and $f\in R(x)$ with $\deg_x f= nm$, there is $h\in R(x)$ such that $f=g\circ h$ if and  only if there is $[a_{p,0}:\ldots:a_{p,nm}:a_{q,0}:\ldots:a_{q,nm}]\in \Zc_R\left(\left\{r_{1,j}\right\}_{j=1}^{n_1}\right)\cap \left(\Pro^{2nm+1}_R\setminus \Zc_R\left(\left\{s_{1,j}\right\}_{j=1}^{n_1}\right)\right)$ such that 
   \begin{equation*}
   f(x)=\frac{\sum_{i=0}^{nm}a_{p,i}x^i}{\sum_{i=0}^{nm}a_{q,i}x^i}.
   \end{equation*}
   \end{enumerate}
    \item[ii)]There is a family of polynomials  $\Gc^{(2)}_{g,n}=\{r_{2,j}\}_{j=1}^{n_2}\cup \{s_{2,j}\}_{j=1}^{m_2}$ in \\$\Com[z_0,z_1,\ldots, z_{nm},w_0,w_1,\ldots,w_{nm}]$. With the following two properties.
   \begin{enumerate}
   \item[1)]\begin{equation*}
   \max_{1\leq i\leq n_1}\deg r_{2,i}\leq 2\left(\frac{3}{2}\right)^{2^{4nm+3}}.
   \end{equation*}
   \item[2)]For all $R\in \Pc_\field$, $x$ algebraically independent variable over $R$  and $f\in R(x)$ with $\deg_x f= nm$, there is $h\in R(x)$ such that $f=h\circ g$ if and  only if there is $[a_{p,0}:\ldots:a_{p,nm}:a_{q,0}:\ldots:a_{q,nm}]\in \Zc_R\left(\left\{r_{2,j}\right\}_{j=1}^{n_2}\right)\cap \left(\Pro^{2nm+1}_R\setminus \Zc_R\left(\left\{s_{2,j}\right\}_{j=1}^{n_2}\right)\right)$ such that 
   \begin{equation*}
   f(x)=\frac{\sum_{i=0}^{nm}a_{p,i}x^i}{\sum_{i=0}^{nm}a_{q,i}x^i}.
   \end{equation*}
\end{enumerate}      
\end{enumerate} 
  \end{thm}
  \begin{proof}
  The proofs of i) and ii) are symmetric so ii) can be proven \emph{mutatis mutandis} the proof of i); thus we only show i). From Remark \ref{R24}, $\Fc^{(1)}_{g,n}$ is a family of homogeneous polynomials of degree $m$. From Remark \ref{R26}, there is not a common zero of the polynomials $q_{1,0},q_{1,1},\ldots,q_{1,nm}$ in $U_{\Com,n}$. Hence  $\phi^{(1)}_{g,\Com,n}$ satisfy the assumptions of Corollary \ref{R12}, and thereby there exists $\Gc^{(1)}_{g,n}:=\{r_{1,j}\}_{j=1}^{n_1}\cup \{s_{1,j}\}_{j=1}^{m_1}$ in $\Com[z_0,z_1,\ldots, z_{nm},w_0,w_1,\ldots,w_{nm}]$ such that
  \begin{equation}
  \label{E22}
   \max_{1\leq i\leq n_1}\deg r_{1,i}\leq2\left(\frac{m^2}{2}+m\right)^{2^{2nm+2n+3}} \leq 2\left(\frac{n^2m^2}{2}+nm\right)^{2^{4nm+3}}
  \end{equation}
  and 
  \begin{equation}
  \label{E23}
 \phi^{(1)}_{g,\Com,n}(U_{\Com,n})=\Zc_\Com\left(\left\{r_{1,j}\right\}_{j=1}^{n_1}\right)\cap \left(\Pro^{2nm+1}_\Com\setminus \Zc_\Com\left(\left\{s_{1,j}\right\}_{j=1}^{n_1}\right)\right).
  \end{equation}
  We shall show that $\Gc^{(1)}_{g,n}$ is the family we are looking for. First note that 1) follows from (\ref{E22}); thus it remains to prove 2). We claim that for any $R\in \Pc_\field$,
  \begin{equation}
  \label{E24}
   \phi^{(1)}_{g,R,n}(U_{R,n})=\Zc_R\left(\left\{r_{1,j}\right\}_{j=1}^{n_1}\right)\cap \left(\Pro^{2nm+1}_R\setminus \Zc_R\left(\left\{s_{1,j}\right\}_{j=1}^{n_1}\right)\right).
  \end{equation}
  Indeed, if $\field=\Com$, then we extend the basis by $R$ (i.e. $X\mapsto X\times_{\mathrm{Spec }\Com}\mathrm{Spec }R$) for the sets $U_{\Com,n}, \Zc_\Com\left(\left\{r_{1,j}\right\}_{j=1}^{n_1}\right), \Pro^{2nm+1}_\Com\setminus \Zc_\Com\left(\left\{s_{1,j}\right\}_{j=1}^{n_1}\right)$ and the morphism $\phi^{(1)}_{g,\Com,n}$, and then (\ref{E23}) implies (\ref{E24}). If $\field=\Real$, take $R':=R\otimes _\Real\Com$; the definition of $\Pc_\Real$ implies that there are  algebraically independent variables $x_1,x_2,\ldots,x_k$  over $\Real$ such that $R=\Real[x_1,x_2,\ldots,x_k]$ so 
  $R'=\Com[x_1,x_2,\ldots,x_k]$. The natural embedding $\Real\hookrightarrow \Com$ induces the embeddings $\Pro^{2n+1}_R\hookrightarrow \Pro^{2n+1}_{R'}$ and $\Pro^{2nm+1}_R\hookrightarrow \Pro^{2nm+1}_{R'}$. In the complex case, we already saw that 
  \begin{equation}
  \label{E25}
     \phi^{(1)}_{g,R',n}(U_{R',n})=\Zc_{R'}\left(\left\{r_{1,j}\right\}_{j=1}^{n_1}\right)\cap \left(\Pro^{2nm+1}_{R'}\setminus \Zc_{R'}\left(\left\{s_{1,j}\right\}_{j=1}^{n_1}\right)\right),
  \end{equation}
  but also 
  \begin{align}
  \label{E26}
  \phi^{(1)}_{g,R,n}(U_{R,n})&=\phi^{(1)}_{g,R',n}(U_{R',n})\cap \Pro^{2nm+1}_R\nonumber\\
  \Zc_{R}\left(\left\{r_{1,j}\right\}_{j=1}^{n_1}\right)&=\Zc_{R'}\left(\left\{r_{1,j}\right\}_{j=1}^{n_1}\right)\cap\Pro^{2nm+1}_R\nonumber\\
  \Pro^{2nm+1}_{R}\setminus \Zc_{R}\left(\left\{s_{1,j}\right\}_{j=1}^{n_1}\right)&=\left(\Pro^{2nm+1}_{R'}\setminus \Zc_{R'}\left(\left\{s_{1,j}\right\}_{j=1}^{n_1}\right)\right)\cap \Pro_{R}^{2nm+1}.
  \end{align}
  Thus (\ref{E25}) and (\ref{E26}) implies (\ref{E24}) when $\field=\Real$, and therefore (\ref{E24}) holds in any case. Assume that $f=g\circ h$. Then
  \begin{equation*}
  \deg h=\frac{\deg f}{\deg g}=n.
  \end{equation*}
  From Lemma \ref{R25} i), there is $\bb\in U_{R,n}$ such that 
 \begin{equation}
 \label{E27}
 f(x)=\frac{\sum_{i=0}^{nm}p_{1,i}(\bb)x^i}{\sum_{i=0}^{nm}q_{1,i}(\bb)x^i}
 \end{equation}
 Hence (\ref{E27}) leads to the existence of $\ab=[a_{p,0}:\ldots:a_{p,nm}:a_{q,0}:\ldots:a_{q,nm}] \in  \phi^{(1)}_{g,R,n}(U_{R,n})$ such that
 \begin{equation*}
 f(x)=\frac{\sum_{i=0}^{nm}a_{p,i}x^i}{\sum_{i=0}^{nm}a_{q,i}x^i},
 \end{equation*}
 and then $\ab\in \Zc_R\left(\left\{r_{1,j}\right\}_{j=1}^{n_1}\right)\cap \left(\Pro^{2nm+1}_R\setminus \Zc_R\left(\left\{s_{1,j}\right\}_{j=1}^{n_1}\right)\right)$ by (\ref{E24}). Now assume that there is $\ab=[a_{p,0}:\ldots:a_{p,nm}:a_{q,0}:\ldots:a_{q,nm}]\in \Zc_R\left(\left\{r_{1,j}\right\}_{j=1}^{n_1}\right)\cap \left(\Pro^{2nm+1}_R\setminus \Zc_R\left(\left\{s_{1,j}\right\}_{j=1}^{n_1}\right)\right)$ such that 
 \begin{equation*}
 f(x)=\frac{\sum_{i=0}^{nm}a_{p,i}x^i}{\sum_{i=0}^{nm}a_{q,i}x^i},
 \end{equation*}
 From (\ref{E24}), $\ab \in \phi^{(1)}_{g,R,n}(U_{R,n})$ so there is $\bb\in U_{R,n}$ such that $\phi^{(1)}_{g,R,n}(\bb)=\ab$ and thereby
 \begin{equation}
 \label{E28}
 f(x)=\frac{\sum_{i=0}^{nm}p_{1,i}(\bb)x^i}{\sum_{i=0}^{nm}q_{1,i}(\bb)x^i}.
 \end{equation}
 Finally, Lemma \ref{R25} i) and (\ref{E28}) imply the existence of $h\in R(x)$ with $\deg h\leq n$ such that $f=g\circ h$.
  \end{proof}
  For any $d\in\Pos$, we define
  \begin{align*}
  c_{3,d}:=(2d^2)^{\max\{2,d\}^{4d+5}}.
  \end{align*}
   \begin{prop}
  \label{R28}
  Let $\field\in\{\Real,\Com\}$, $f\in\field(x_1,x_2)$ with $d:=\deg f$ , $g\in\field(x)$ and $A$ be a subset of $\field$.
  \begin{enumerate}
  \item[i)]Assume that for all $a\in A$, there is $h_a\in \field(x_1)$ such that $f(x_1,a)=g(h_a(x_1))$.  If $|A|\geq c_{3,d}$, then there is $h\in\field(x_1,x_2)$ such that $f(x_1,x_2)=g(h(x_1,x_2))$.
  \item[ii)]Assume that for all $a\in A$, there is $h_a\in \field(x)$ such that $f(x_1,a)=h_a(g(x_1))$.  If $|A|\geq c_{3,d}$, then there is $h\in\field(x_1,x_2)$ such that $f(x_1,x_2)=h(g(x_1),x_2)$.
  \end{enumerate}
  \end{prop}
  \begin{proof}
  Set $m:=\deg g$ and $d_1:=\deg_{x_1} f$. Write $f=\frac{p}{q}$ with $p,q\in\field[x_1,x_2]$ and $(p,q)=1$; also write $p(x_1,x_2)=\sum_{i=0}^{d_1}p_i(x_2)x_1^i$ and $q(x_1,x_2)=\sum_{i=0}^{d_1}q_i(x_2)x_1^i$ where $p_0,p_1,\ldots, p_{d_1},q_0,\ldots, q_{d_1}\in\field[x_2]$. Since $d_1=\deg_{x_1}f$, we have that $p_{d_1}\neq 0$ or $q_{d_1}\neq 0$; assume without loss of generality that $p_{d_1}\neq 0$. Set
  \begin{align*}
  A_1&:=\{a\in A:\;p_{d_1}(a)=0\}\\
  A_2&:=\{a\in A:\;(p(x_1,a),q(x_1,a))\neq 1\}\\
  A'&:=A\setminus (A_1\cup A_2).
  \end{align*}
  On the one hand, since $\deg_{x_2} p_{d_1}\leq d$, 
  \begin{equation}
  \label{E29}
  |A_1|\leq d.
  \end{equation}
  On the other hand, Corollary \ref{R15} yields that 
  \begin{equation}
  \label{E30}
  |A_2|\leq d^{2d}.
  \end{equation}
  Then (\ref{E29}) and (\ref{E30}) imply that
  \begin{equation}
  \label{E31}
  |A'|\geq |A|-d-d^{2d}\geq c_{3,d}-d-d^{2d}>2d\left(\frac{d^2}{2}+d\right)^{2^{4d+3}}.
  \end{equation}
  Now we start with the proof of i). Set  $n:=\frac{d_1}{m}$. For any $a\in A'$, we have that $(p(x_1,a),q(x_1,a))=1$ and  $p_{d_1}(a)\neq 0$ so $\deg_{x_1}f(x_1,a)=d_1$; thus 
    \begin{equation*}
    \deg_{x_1} h_a=\frac{\deg_{x_1} f(x_1,a)}{\deg_{x_1}g(x_1)}=\frac{d_1}{m}=n.
  \end{equation*}
  Let $\Gc^{(1)}_{g,n}=\{r_{1,j}\}_{j=1}^{n_1}\cup \{s_{1,j}\}_{j=1}^{m_1}$ be the family of polynomials found in Theorem \ref{R27} and define
  \begin{equation*}
  \bb(x_2):=[p_0(x_2):p_1(x_2):\ldots:p_{d_1}(x_2):q_0(x_2):\ldots:q_{d_1}(x_2)]
  \end{equation*}
    Take $a\in A'$. Since $f(x_1,a)=g(h_a(x_1))$, Theorem \ref{R27} i) implies the existence of $\ab=[a_{p,0}:\ldots:a_{p,nm}:a_{q,0}:\ldots:a_{q,nm}]\in \Zc_\field\left(\left\{r_{1,j}\right\}_{j=1}^{n_1}\right)\cap \left(\Pro^{2nm+1}_\field\setminus \Zc_\field\left(\left\{s_{1,j}\right\}_{j=1}^{n_1}\right)\right)$ such that 
   \begin{equation}
   \label{E32}
   f(x_1,a)=\frac{\sum_{i=0}^{nm}a_{p,i}x_1^i}{\sum_{i=0}^{nm}a_{q,i}x_1^i}.
   \end{equation}
  On the one hand, $(p(x_1,a),q(x_1,a))=1$ and  $f(x_1,a)=\frac{p(x_1,a)}{q(x_1,a)}$; thus (\ref{E32}) yields that there is $r\in\field[x_1]$ such that $\sum_{i=0}^{nm}a_{p,i}x_1^i=r(x_1)p(x_1,a)$ and  $\sum_{i=0}^{nm}a_{q,i}x_1^i=r(x_1)q(x_1,a)$. On the other hand, 
  \begin{align*}
 \max\left\{\deg_{x_1} \sum_{i=0}^{nm}a_{p,i}x_1^i,\deg_{x_1} \sum_{i=0}^{nm}a_{q,i}x_1^i\right\}&\leq nm\\
 &=d_1\\
 &=\deg_{x_1} f(x_1,a)\\
 &=\max\{\deg_{x_1}p(x_1,a),\deg_{x_1}q(x_1,a)\},
\end{align*}   
   so $\deg_{x_1}r(x_1)=0$ and therefore $\bb(a)=\ab$. Hence we have proven that for all $a\in A'$,
   \begin{equation}
   \label{E33}
   \bb(a)\in \Zc_\field\left(\left\{r_{1,j}\right\}_{j=1}^{n_1}\right)\cap \left(\Pro^{2d_1+1}_\field\setminus \Zc_\field\left(\left\{s_{1,j}\right\}_{j=1}^{n_1}\right)\right).
   \end{equation}
   
   On the one hand, since 
   \begin{equation*}
   \max\{\deg_{x_2}p_0,\ldots,\deg_{x_2}p_{d_1},\deg_{x_2}q_0,\ldots,\deg_{x_2}q_{d_1}\}\leq d
   \end{equation*}
   and 
   \begin{equation*}
   \max_{1\leq i\leq n_1}\deg r_{1,i}\leq 2\left(\frac{d^2}{2}+d\right)^{2^{4d+3}},
   \end{equation*}
  we get that  for each $i\in\{1,2,\ldots, n_1\}$,   
   the polynomial $r_{1,i}(\bb(x_2))=\\r_{1,i}(p_0(x_2),\ldots,p_{d_1}(x_2),q_0(x_2),\ldots,q_{d_1}(x_2))$ has degree at most $d\cdot 2\left(\frac{d^2}{2}+d\right)^{2^{4d+3}}$; however, (\ref{E33}) implies that $r_{1,i}(\bb(x_2))$ has at least $|A'|$ roots, and (\ref{E31}) yields that $|A'|>d\cdot 2\left(\frac{d^2}{2}+d\right)^{2^{4d+3}}$ so $r_{1,i}(\bb(x_2))$ is the zero polynomial in $\field[x_2]$.  Hence
   \begin{equation}
   \label{E34}
   \bb(x_2)\in \Zc_{\field[x_2]}\left(\left\{r_{1,j}\right\}_{j=1}^{n_1}\right).
   \end{equation}
   
   On the other hand, for each $i\in\{1,2,\ldots, m_1\}$,   
   the polynomial $s_{1,i}(\bb(x_2))=\\s_{1,i}(p_0(x_2),\ldots,p_{d_1}(x_2),q_0(x_2),\ldots,q_{d_1}(x_2))$ is nonzero in $A'$ by (\ref{E33}) so it is not the zero polynomial in $\field[x_2]$; thereby
   \begin{equation}
   \label{E35}
   \bb(x_2)\in \Pro^{2d_1+1}_{\field[x_2]}\setminus \Zc_{\field[x_2]}\left(\left\{s_{1,j}\right\}_{j=1}^{n_1}\right).
   \end{equation}
      From  (\ref{E34}) and (\ref{E35}),
   \begin{equation*}
   \bb(x_2)\in \Zc_{\field[x_2]}\left(\left\{r_{1,j}\right\}_{j=1}^{n_1}\right)\cap \Pro^{2d_1+1}_{\field[x_2]}\setminus \Zc_{\field[x_2]}\left(\left\{s_{1,j}\right\}_{j=1}^{n_1}\right).
   \end{equation*}
  Thus, from Theorem \ref{R27} i),  we conclude there is $h\in (\field[x_2])(x_1)\simeq \field (x_1,x_2)$ such that $f(x_1,x_2)=g(h(x_1,x_2))$. 
   
   The proof of ii) is symmetric to the proof of i) and it can be done \emph{mutatis mutandis}.
  \end{proof}
  \begin{lem}
  \label{R29}
  Let $\field\in\{\Real,\Com\}$. For $i\in\{1,2\}$, let $p_{i,1},p_{i,2},q_{i,1},q_{i,2}\in\field[x_i]$ be such that $(p_{i,1},q_{i,1})=1$ and $(p_{i,2},q_{i,2})=1$, and set $f_{i,1}:=\frac{p_{i,1}}{q_{i,1}}, f_{i,2}:=\frac{p_{i,2}}{q_{i,2}}\in\field(x_i)$. Define $r_1(x_1,x_2):=p_{1,1}(x_1)q_{2,1}(x_2)-p_{2,1}(x_2)q_{1,1}(x_1)$ and $r_2(x_1,x_2):=p_{1,2}(x_1)q_{2,2}(x_2)-p_{2,2}(x_2)q_{1,2}(x_1)$ polynomials in $\field[x_1,x_2]$. Assume that there is a nonconstant irreducible $r\in\field[x_1,x_2]$ which divides $r_1$ and $r_2$. Then there are $g_1,g_2,h_1,h_2\in\field(x)$ such that
  \begin{align*}
  f_{1,1}(x_1)&=g_1(h_1(x_1))&f_{2,1}(x_2)=g_1(h_2(x_2))\\
   f_{1,2}(x_1)&=g_2(h_1(x_1))&f_{2,2}(x_2)=g_2(h_2(x_2)).
  \end{align*}
  \end{lem}
  \begin{proof}
  Since $r$ is not constant, $\deg_{x_1}r\geq 1$ or $\deg_{x_2}r\geq 1$; assume without  loss of generality that $\deg_{x_1}r\geq 1$. From Theorem \ref{R6}, the field $\field\subseteq \field(f_{1,1}(x_1), f_{1,2}(x_1))\\\subseteq \field(x_1)$ is a simple extension so there is $h_1\in\field(x)$ such that 
  \begin{equation*}
\field(f_{1,1}(x_1), f_{1,2}(x_1))= \field(h_1(x_1)).
  \end{equation*}
  Therefore there are $g_1,g_2\in\field(x)$ and $g\in\field(y_1,y_2)$ such that $f_{1,1}(x_1)=g_1(h_1(x_1)),\\ f_{1,2}(x_1)=g_2(h_1(x_1))$ and $h_1(x_1)=g(f_{1,1}(x_1),f_{1,2}(x_1))$. Define $h_2(x_2):=\\g(f_{2,1}(x_2),f_{2,2}(x_2))$. For any $p\in\field[x_1,x_2]\setminus\{0\}$, let $v_r(p)$ be the biggest $n\in\Nat $ such that $r^n$ divides $p$. The valuation is extended to $\field(x_1,x_2)$  taking $v_r\left(\frac{p}{q}\right)=v_r(p)-v_r(q)$ for $p,q\in\field[x_1,x_2]\setminus \{0\}$, and $v_r(0)=\infty$. For any $f_1,f_2\in\field(x_1,x_2)$, we write $f_1\equiv f_2$ if  $v_r(f_1-f_2)>0$.   We claim that for $i,j\in\{1,2\}$,  neither $q_{i,j}(x_i)$ nor $p_{i,j}(x_i)$ is divisible by $r(x_1,x_2)$. We prove this claim for $q_{1,1}$ and the other cases are done likewise.  If $r$ divides $q_{1,1}(x_1)$, then $r$ divides $q_{1,1}(x_1)p_{2,1}(x_2)$; therefore, since $r$ divides $r_1$ , $r$ divides $p_{1,1}(x_1)q_{2,1}(x_2)$; insomuch as $\deg_{x_1}r\geq 1$, we conclude that $r$ cannot divide $q_{2,1}(x_2)$ so $r$ has a common factor with $p_{1,1}(x_1)$; however this is impossible since $(p_{1,1},q_{1,1})=1$. Therefore,  since $r$ divides $r_1$ and $r_2$, we get that
   \begin{align}
   \label{E36}
   f_{1,1}(x_1)&\equiv f_{2,1}(x_2)& f_{1,2}(x_1)\equiv f_{2,2}(x_2).
   \end{align}
   Then
   \begin{align}
   \label{E37}
   h_2(x_2)&=g(f_{2,1}(x_2),f_{2,2}(x_2))\nonumber\\
      &\equiv g(f_{1,1}(x_1),f_{1,2}(x_1))&\Big(\text{by (\ref{E36})}\Big)\nonumber\\
       &=h_1(x_1).
   \end{align}
   Moreover, for $j\in\{1,2\}$, 
   \begin{align}
   \label{E38}
   f_{2,j}(x_2)&\equiv f_{1,j}(x_1)&\Big(\text{by (\ref{E36})}\Big)\nonumber\\
   &=g_j(h_1(x_1))\nonumber\\
   &\equiv g_j(h_2(x_2)).&\Big(\text{by (\ref{E37})}\Big)
   \end{align}
  We conclude the proof showing that not only $f_{2,j}\equiv g_j(h_2(x_2))$ but also $f_{2,j}=g_j(h_2(x_2))$ for $j\in\{1,2\}$.  Let $r_1,r_2,s_1,s_2\in \field[x_2]$ be such that $g_1(h_2(x_2))=\frac{r_1(x_2)}{s_1(x_2)}$ and $g_2(h_2(x_2))=\frac{r_2(x_2)}{s_2(x_2)}$ with $(r_1,s_1)=(r_2,s_2)=1$. For $j\in\{1,2\}$, on the one hand, $r_j(x_2)q_{2,j}(x_2)-s_j(x_2)p_{2,j}(x_2)\in\field[x_2]$; on the other hand, (\ref{E38}) implies that $r_j(x_2)q_{2,j}(x_2)-s_j(x_2)p_{2,j}(x_2)$ is in the ideal $\left<r(x_1,x_2)\right>$ of $\field[x_1,x_2]$; thus
   \begin{equation}
   \label{E39}
   r_j(x_2)q_{2,j}(x_2)-s_j(x_2)p_{2,j}(x_2)\in\field[x_2]\cap \left<r(x_1,x_2)\right>.
   \end{equation}
   However, since $\deg_{x_1}r\geq 1$, any nonzero multiple of $r$ has degree at least $1$ with respect to $x_1$ so 
   \begin{equation*}
   \field[x_2]\cap \left<r(x_1,x_2)\right>=\{0\}.
   \end{equation*}
   Then (\ref{E39}) implies that  $f_{2,1}(x_2)=g_1(h_2(x_2))$ and $ f_{2,2}(x_2)=g_2(h_2(x_2))$, and this concludes the proof. 
  \end{proof}
  \begin{lem}
  \label{R30}
  Let $f\in\Com(x_1,x_2)$ and $A$ be a subset of $\Com$ such that $|A|> 2\deg f+(\deg f)^{\deg f}$. Then 
  \begin{align*}
  \deg_{x_1}f&=\max_{a\in A}\deg_{x_1}f(x_1,a)\\
   \deg_{x_2}f&=\max_{a\in A}\deg_{x_2}f(a,x_2).
  \end{align*}
  \end{lem}
  \begin{proof}
  Write $f=\frac{p}{q}$ with $p,q\in\field[x_1,x_2]$ and $(p,q)=1$. Let $p_0,p_1,\ldots, p_{d_p},q_0,\ldots,\\ q_{d_q}\in\Com[x_2]$ be such that $p_{d_p}\cdot q_{d_q}\neq 0$, $p(x_1,x_2)=\sum_{i=0}^{d_p}p_i(x_2)x_1^i $ and $q(x_1,x_2)=\sum_{i=0}^{d_q}q_i(x_2)x_1^i$.  Define 
\begin{equation*}
A':=\{a\in A:\;(p(x_1,a),q(x_1,a))=1\}
\end{equation*}  
  On the one hand,
  \begin{equation}
  \label{E40}
  \deg_{x_1}f=\max\{d_p,d_q\}.
  \end{equation}
  On the other hand, for each $a\in A$
  \begin{align}
  \label{E0}
  \deg_{x_1}f(x_1,a)&\leq \max\{\deg_{x_1}p(x_1,a),\deg_{x_1}q(x_1,a)\}\nonumber\\
  &=\max\{j\in\Nat:\; p_j(a)\neq 0 \text{ or }q_j(a)\neq 0\}.
  \end{align}
   For  $a\in A'$, we have that $(p(x_1,a),q(x_1,a))=1$ so
  \begin{align}
  \label{E41}
  \deg_{x_1}f(x_1,a)&=\max\{\deg_{x_1}p(x_1,a),\deg_{x_1}q(x_1,a)\}\nonumber\\
  &=\max\{j\in\Nat:\; p_j(a)\neq 0 \text{ or }q_j(a)\neq 0\}.
  \end{align}
  Corollary \ref{R15} implies that $|A\setminus A'|\leq (\deg f)^{\deg f}$ so
  $|A'|>2\deg f\geq 2\deg_{x_2} f$. Hence we get that there is $a_0\in A'$ such that $p_{d_p}(a_0)\neq 0$ and $q_{d_q}(a_0)\neq 0$. Therefore (\ref{E41}) leads to
  \begin{equation}
  \label{E42}
  \deg_{x_1}f(x_1,a_0)=\max\{d_p,d_q\}.
  \end{equation}
  Thus 
  \begin{align*}
  \deg_{x_1}f&=\max\{d_p,d_q\}&\Big(\text{by (\ref{E40})}\Big)\\
  &=\deg_{x_1}f(x_1,a_0)&\Big(\text{by (\ref{E42})}\Big)\\
  &=\max_{a\in A}\deg_{x_1}f(x_1,a).&\Big(\text{by (\ref{E0}),(\ref{E41})}\Big)
  \end{align*}
  Likewise we get that 
   \begin{equation*}
  \deg_{x_2}f=\max_{a\in A}\deg_{x_2}f(a,x_2).\qedhere
  \end{equation*}
  \end{proof}
  \section{Dominating functions}
  Let $\field\in\{\Real,\Com\}$ and $g,f_1,f_2\in\field(x)$. We say that $g$ \emph{dominates} $\{f_1,f_2\}$ if the following two conditions are satisfied:
  \begin{enumerate}
  \item[$\bullet$] There exist $h_1,h_2\in\field(x)$ such that $f_1=g\circ h_1$ and $f_2=g\circ h_2$.
  \item[$\bullet$] For any $g',h'_1,h'_2\in\field(x)$ such that $f_1=g'\circ h'_1$ and $f_2=g'\circ h'_2$, we have that $\deg g'\leq \deg g$.
  \end{enumerate}  
  
  Note that for any pair $\{f_1,f_2\}$ in $\field(x)$, there exists  $g\in\field(x)$ which dominates it. Recall that $g,g'\in\field(x)$ are equivalent if there is $h\in\field(x)$ with $\deg h=1$ such that $g=g'\circ h$. 
  \begin{rem}
  \label{R31}
  Let $g,g'\in\field(x)$ be equivalent. For all $f_1,f_2,\in\field(x)$, we note that $g$ dominates $\{f_1,f_2\}$ if and only if $g'$ dominates $\{f_1,f_2\}$.
  \end{rem}
  The following lemma is based  on the idea of the proof of \cite[Thm. 18]{ER}.
  \begin{lem}
  \label{R32}
  Let $\field\in\{\Real,\Com\}$, $d\in\Nat$ and $f\in\field(x_1,x_2)$ be such that $\deg f\leq d$. 
  \begin{enumerate}
  \item[i)] There is $c_{4,d}>0$ with the following property. For any finite subset $A$ of $\field$ with $|A|\geq 2$, there are $g\in\field(x_1)$ and $B\subseteq A\times A$ such that  $|B|\geq c_{4,d}|A|^2$ and $g$ dominates $\{f(x_1,a), f(x_1,a')\}$ for all $(a,a')\in B$.
   \item[ii)] There is $c_{5,d}>0$ with the following property. For any finite subset $A$ of $\field$ with $|A|\geq 2$, there are $g\in\field(x_2)$ and $B\subseteq A\times A$ such that  $|B|\geq c_{5,d}|A|^2$ and $g$ dominates $\{f(a,x_2), f(a',x_2)\}$ for all $(a,a')\in B$.
  \end{enumerate}
  \end{lem}
  \begin{proof}
  Recall that $c_2=c_{2,c,n}$ is the number that satisfies Lemma \ref{R19} for the parameters $c$ and $n$. Take $c_{4,d}:=\frac{c_2\left(\frac{1}{2},2^d\right)}{2}$. Let $G=(V,E)$ be the graph with $V=A$ and $E=\left\{\{a,a'\}:\;a,a'\in A, \, a\neq a'\right\}$. For each $\{a,a'\}\in E$, let $g_{a,a'}$ be a function that dominates $\{f(x_1,a), f(x_1,a')\}$. For each equivalence class of $\field(x)$, take a representative and denote by $\{g_i\}_{i\in I}$ the set of representatives of all the classes. We colour $E$ as follows. For each $i\in I$, $E_i$ is the set of $\{a,a'\}\in E$ such that $g_{a,a'}$ is equivalent to $g_i$; denote the colouring by $E=\bigcup_{i\in I}E_i$. For any $a\in A=V$, $f(x_1,a)$ has at most $2^d$ nonequivalent decompositions by Proposition \ref{R18}. This means that at most $2^d$ colors meet in each vertex of $G$. Also note that $G$ is complete so the degree of each vertex is $|A|-1\geq \frac{|A|}{2}$. Thus the assumptions of Lemma \ref{R19} are satisfied by $G$, and therefore there is a monochromatic subgraph $G'=(V',E')$ of $G$ such that $|E'|\geq 2c_{4,d} |E|$. Since $G'$ is monochromatic, there is $g\in\field(x)$ such that $g$ is equivalent to $g_{a,a'}$ for all $\{a,a'\}\in E'$. Therefore $  B:=\{(a,a')\in A:\;\{a,a'\}\in E'\}$ satisfies that 
  \begin{equation*}
  |B|\geq 2|E'|\geq 4c_{4,d}|E|= 2c_{4,d}(|A|^2-|A|)\geq c_{4,d}|A|^2
  \end{equation*}
  and $g$ dominates $\{f(x_1,a), f(x_1,a')\}$ for all $(a,a')\in B$; this proves i). The proof of ii) is done in the same way.
 \end{proof}
 For $d\in\Nat$, set
 \begin{equation*}
 c_{6,d}:=2^d\cdot c_{3,d}=2^d\cdot (2d^2)^{\max\{2,d\}^{4d+5}}.
 \end{equation*}
 \begin{lem}
 \label{R33}
  Let $\field\in\{\Real,\Com\}$, $d\in\Nat$ and $f\in\field(x_1,x_2)$ be such that $\deg f\leq d$. There is $c_{7,d}>0$ with the following property.  For any finite subsets $A_1,A_2$ of $\field$ with $|A_1|,|A_2|>\max\left\{\frac{1}{c_{7,d}}, 2\right\}$, there are $h\in \field(x_1,x_2)$, $g\in\field(x)$,  $B_1\subseteq A_1\times A_1$ and $B_2\subseteq A_2\times A_2$ such that the following conditions are satisfied.
\begin{enumerate}
\item[i)]$f(x_1,x_2)=g(h(x_1,x_2))$.
\item[ii)]$|(A_1\times A_1)\setminus B_1|\leq c_{6,d}|A_1|$ and $|(A_2\times A_2)\setminus B_2|\leq c_{6,d}|A_2|$.
\item[iii)]$g(x_1)$ dominates $\{f(x_1,a_2),f(x_1,a'_2)\}$ for all $(a_2,a'_2)\in B_2$ and $g(x_2)$ dominates $\{f(a_1,x_2),f(a'_1,x_2)\}$ for all $(a_1,a'_1)\in B_1$.
\end{enumerate}  
    \end{lem}
    \begin{proof}
    Let $c_{4,d}$ and $c_{5,d}$ be as  in Lemma \ref{R32}. Take $c_{7,d}:=\frac{\min\{c_{4,d},c_{5,d}\}}{c_{3,d}}$. From Lemma \ref{R32}, there are $g_1\in\field(x_1)$ (resp. $g_2\in\field(x_2)$) and $B'_2\subseteq A_2\times A_2$ (resp.  $B'_1\subseteq A_1\times A_1$) such that  $|B'_2|\geq c_{4,d}|A_2|^2$ (resp. $|B'_1|\geq c_{5,d}|A_1|^2$) and $g_1$ (resp. $g_2$) dominates $\{f(x_1,a), f(x_1,a')\}$ (resp. $\{f(a,x_2), f(a',x_2)\}$) for all $(a,a')\in B'_2$ (resp. $(a,a')\in B'_1$). Without loss of generality, assume that $\deg g_2\geq \deg g_1$ and write $g:=g_2$.
    
    Let $B_1$ be the subset of elements  $\ab_1=(a_1,a'_1)$ of $ A_1\times A_1$ such that $g(x_2)$ dominates $\{f(a_1,x_2),f(a'_1,x_2)\}$. Since $|B'_1|\geq c_{5,d}|A_1|^2$ and $B'_1\subseteq B_1$, there are  $a_1\in A_1$ and $A'_1$ a subset of $A_1$ such that $|A'_1|\geq c_{5,d}|A_1|$ and $(a_1,a'_1)\in B_1$ for all $a'_1\in A'_1$. For all $a'_1\in A'_1$, we have that $(a_1,a'_1)\in B_1$ and therefore there are $h_{a'_1},h_{(a_1,a'_1)}\in\field(x_2)$ such that $f(a_1,x_2)=g(h_{(a_1,a'_1)}(x_2))$ and $f(a'_1,x_2)=g(h_{a'_1}(x_2))$. Since $|A'_1|\geq c_{5,d}|A_1|>c_{3,d}$ and $f(a'_1,x_2)=g(h_{a'_1}(x_2))$ for all $a_1'\in A'_1$, Proposition \ref{R28} i) implies that there is $h\in\field(x_1,x_2)$ such that
\begin{equation}  
\label{E43}  
     f(x_1,x_2)=g(h(x_1,x_2)).
\end{equation}
 For each $a\in A_1$, write $A_{1,a}:=\{(a,a')\in A_1\times A_1:\;(a,a')\in B_1\}$. We claim that for each $a\in A_1$,
    \begin{equation}
    \label{E44}
    |A_1\setminus A_{1,a}|\leq c_{6,d}.
    \end{equation}
    To show (\ref{E44}), we assume that it is false and we get a contradiction. Insomuch as $f(x_1,x_2)=g(h(x_1,x_2))$, we have that $f(a',x_2)=g(h(a',x_2))$ for all $a'\in A_1$. Thus, for all $a'\in A_1\setminus A_{1,a}$, there are $g_{a'}\in\field(x)$ such that $f(a,x_2)$ and $f(a',x_2)$ factorizes through $g_{a'}$ with $\deg g_{a'}>\deg g$; hence there are $h_{a'},h_{(a,a')}\in\field(x_2)$ such that $f(a,x_2)=g_{a'}(h_{(a,a')}(x_2))$ and $f(a',x_2)=g_{a'}(h_{a'}(x_2))$. From Proposition \ref{R18}, $f(a,x_2)$ has at most $2^d$ nonequivalent decompositions. Hence, if (\ref{E44}) is false, there is $A'\subseteq A_1\setminus A_{1,a}$ such that $|A'|>\frac{c_{6,d}}{2^d}=c_{3,d}$ and the decompositions $f(a,x_2)=g_{a'}(h_{(a,a')}(x_2))$ are linearly equivalent for all $a'\in A'$; thus we may assume that all of them are equal. Then set  $g':=g_{a'}$ for any $a'\in A'$ and note that
\begin{equation}
\label{E45}
\deg g'>\deg g.
\end{equation}    
     Since $f(a,x_2)=g'(h_{(a,a')}(x_2))$ for all $a'\in A'$ and $|A'|>c_{3,d}$, Proposition \ref{R28} i) implies that there is $h'\in\field(x_1,x_2)$ such that $f(x_1,x_2)=g'(h'(x_1,x_2))$. Now, for any $(a',a'')\in B_1$, we have that  $f(a',x_2)=g'(h'(a',x_2))$ and $f(a'',x_2)=g'(h'(a'',x_2))$; however, (\ref{E45}) implies that $g$ does not dominate $\{f(a',x_2),f(a'',x_2)\}$ which is impossible by the definition of $B_1$. This proves (\ref{E44}).  From (\ref{E44}), we conclude that 
     \begin{equation}
     \label{E46}
     |(A_1\times A_1)\setminus B_1|\leq \sum_{a\in A_1}|A_1\setminus A_{1,a}|\leq c_{6,d}|A_1|.
     \end{equation}
     
     Let $B_2$ be the subset of elements  $\ab_2=(a_2,a'_2)$ of $ A_2\times A_2$ such that $g(x_1)$ dominates $\{f(x_1,a_2),f(x_1,a'_2)\}$.    For each $a\in A_2$, write $A_{2,a}=\{(a,a')\in A_2\times A_2:\;(a,a')\in B_2\}$. We claim that
    \begin{equation}
    \label{E47}
    |A_2\setminus A_{2,a}|\leq c_{6,d}.
    \end{equation}
    We show (\ref{E47}) assuming it is false and getting a contradiction. Insomuch as $f(x_1,x_2)\\=g(h(x_1,x_2))$ by (\ref{E43}), we have that $f(x_1,a')=g(h(x_1,a'))$ for all $a'\in A_2$. Thus, for all $a'\in A_2\setminus A_{2,a}$, there are $g_{a'}\in\field(x)$ such that $f(x_1,a)$  and $f(x_1,a')$ factorizes through $g_{a'}$ with $\deg g_{a'}>\deg g$; hence there exist $h_{a'},h_{(a,a')}\in\field(x_1)$ such that $f(x_1,a)=g_{a'}(h_{(a,a')}(x_1))$ and $f(x_1,a')=g_{a'}(h_{a'}(x_1))$. From Proposition \ref{R18}, $f(x_1,a)$ has at most $2^d$ nonequivalent decompositions. Hence, if (\ref{E47}) is false, there is $A'\subseteq A_2\setminus A_{2,a}$ such that $|A'|>\frac{c_{6,d}}{2^d}=c_{3,d}$ and the decompositions $f(x_1,a)=g_{a'}(h_{(a,a')}(x_1))$ are linearly equivalent for all $a'\in A'$; thus all the functions $\{g_{a'}\}_{a'\in A'}$   are equivalent, and, by Remark \ref{R31}, we may assume that all of them are equal to a rational function $g'$; in particular
\begin{equation}
\label{E48}
\deg g'>\deg g\geq \deg g_1.
\end{equation}    
     Since $f(x_1,a)=g(h_{(a,a')}(x_1))$ for all $a'\in A'$ and $|A'|>c_{3,d}$, Proposition \ref{R28} i) implies that there is $h'\in\field(x_1,x_2)$ such that $f(x_1,x_2)=g'(h'(x_1,x_2))$. Now, for any $(a',a'')\in B'_2$, we have that  $f(x_1,a')=g'(h'(x_1,a'))$ and $f(x_1,a'')=g'(h'(x_1,a''))$; however, (\ref{E48}) implies that $g_1$ does not dominate $\{f(x_1,a'),f(x_1,a'')\}$ which is impossible by the election of $B'_2$. This proves (\ref{E47}).  From (\ref{E47}), we conclude that 
     \begin{equation}
     \label{E49}
     |(A_2\times A_2)\setminus B_2|\leq \sum_{a\in A_2}|A_2\setminus A_{2,a}|\leq c_{6,d}|A_2|=c_{6,d}|A_1|.
     \end{equation}
     Finally, (\ref{E43}) yields  i),  (\ref{E46}) and (\ref{E49}) give  ii)  and the definitions of $B_1$ and $B_2$ imply  iii). 
    \end{proof}
    \section{Proof of Theorem \ref{R3}}
    In this section we show Theorem \ref{R3}. We start with a technical lemma. 
     \begin{lem}
    \label{R34}
    Let $\field\in\{\Real, \Com\}$, $d\in\Nat$, $f\in\field(x_1,x_2)$ be such that $\deg f\leq d$, $A$ be a  finite subset of $\field$ and $B$ be a subset of $A\times A$ such that $|B|\geq c_{6,d}^2$.
    \begin{enumerate}
    \item[i)]Assume that for all $\ab=(a_1,a_2),\ab'=(a'_1,a'_2)\in B$ distinct, there are $h^{(\ab)}_{\ab,\ab'}, h^{(\ab')}_{\ab,\ab'},l^{(1)}_{\ab,\ab'}, l^{(2)}_{\ab,\ab'}\in\field(x)$ such that $\deg h^{(\ab)}_{\ab,\ab'}, \deg h^{(\ab')}_{\ab,\ab'}\leq 1$ and  the rational functions $f(y_1,a_1),f(y_1,a'_1)\in\field(y_1)$ and $f(y_2,a_2),f(y_2,a'_2)\in\field(y_2)$ are decomposed as follows
    \begin{align*}
    f(y_1,a_1)&=h^{(\ab)}_{\ab,\ab'}\left(l^{(1)}_{\ab,\ab'}(y_1)\right)& f(y_2,a_2)&=h^{(\ab)}_{\ab,\ab'}\left(l^{(2)}_{\ab,\ab'}(y_2)\right)\\
    f(y_1,a'_1)&=h^{(\ab')}_{\ab,\ab'}\left(l^{(1)}_{\ab,\ab'}(y_1)\right)& f(y_2,a'_2)&=h^{(\ab')}_{\ab,\ab'}\left(l^{(2)}_{\ab,\ab'}(y_2)\right).
    \end{align*}
    Then there are $l\in\field(x)$ and $h\in\field(x_1,x_2)$ such that $\deg_{x_1}h\leq 1$ and $f(x_1,x_2)=h(l(x_1),x_2)$.
     \item[ii)]Assume that for all $\ab=(a_1,a_2),\ab'=(a'_1,a'_2)\in B$ distinct, there are $h^{(\ab)}_{\ab,\ab'}, h^{(\ab')}_{\ab,\ab'},l^{(1)}_{\ab,\ab'}, l^{(2)}_{\ab,\ab'}\in\field(x)$ such that $\deg h^{(\ab)}_{\ab,\ab'}, \deg h^{(\ab')}_{\ab,\ab'}\leq 1$ and  the rational functions $f(y_1,a_1),f(y_1,a'_1)\in\field(y_1)$ and $f(y_2,a_2),f(y_2,a'_2)\in\field(y_2)$ are decomposed as follows
    \begin{align*}
    f(a_1,y_1)&=h^{(\ab)}_{\ab,\ab'}\left(l^{(1)}_{\ab,\ab'}(y_1)\right)& f(a_2,y_2)&=h^{(\ab)}_{\ab,\ab'}\left(l^{(2)}_{\ab,\ab'}(y_2)\right)\\
    f(a'_1,y_1)&=h^{(\ab')}_{\ab,\ab'}\left(l^{(1)}_{\ab,\ab'}(y_1)\right)& f(a'_2,y_2)&=h^{(\ab')}_{\ab,\ab'}\left(l^{(2)}_{\ab,\ab'}(y_2)\right).
    \end{align*}
    Then there are $l\in\field(x)$ and $h\in\field(x_1,x_2)$ such that $\deg_{x_2}h\leq 1$ and $f(x_1,x_2)=h(x_1,l(x_2))$.
    \end{enumerate}
    \end{lem}
   \begin{proof}
   The proofs of i) and ii) are symmetric; thus it is suffices to show i).  Fix $\ab=(a_1,a_2)\in B$. First we claim that  there is $B'\subseteq B$ such that for all $\ab',\ab''\in B'$,  the decompositions
   \begin{align}
   \label{E50}
    f(y_1,a_1)&=h^{(\ab)}_{\ab,\ab'}\left(l^{(1)}_{\ab,\ab'}(y_1)\right)& f(y_1,a_1)=h^{(\ab)}_{\ab,\ab''}\left(l^{(1)}_{\ab,\ab''}(y_1)\right)    \end{align} 
    are linearly equivalent, the decompositions 
     \begin{align}
   \label{E51}
    f(y_2,a_2)&=h^{(\ab)}_{\ab,\ab'}\left(l^{(2)}_{\ab,\ab'}(y_2)\right)& f(y_2,a_2)=h^{(\ab)}_{\ab,\ab''}\left(l^{(2)}_{\ab,\ab''}(y_2)\right)   , \end{align} 
    are linearly equivalent and 
\begin{equation}
\label{E52}
|B'|\geq \frac{1}{2^{2d}}|B|.
\end{equation}    
    From Proposition \ref{R18}, $f(y_1,a_1)$ has at most $2^{\deg_{y_1}f(y_1,a)}\leq 2^d$ nonequivalent decompositions so there is $B''\subseteq B$ such that for all $\ab',\ab''\in B''$, the decompositions $f(y_1,a_1)=h^{(\ab)}_{\ab,\ab'}\left(l^{(1)}_{\ab,\ab'}(y_1)\right)=h^{(\ab)}_{\ab,\ab''}\left(l^{(1)}_{\ab,\ab''}(y_1)\right)$ are equivalent and $|B''|\geq \frac{1}{2^d}|B|$. In the same way $f(y_2,a_2)$ has at most $2^d$ nonequivalent decompositions there is $B'\subseteq B''$ such that for all $\ab',\ab''\in B'$, the decompositions $f(y_2,a_2)=h^{(\ab)}_{\ab,\ab'}\left(l^{(2)}_{\ab,\ab'}(y_2)\right)=h^{(\ab)}_{\ab,\ab''}\left(l^{(2)}_{\ab,\ab''}(y_2)\right)$ are equivalent and $|B'|\geq \frac{1}{2^d}|B''|$; thus $B'$ satisfies (\ref{E50}), (\ref{E51}) and (\ref{E52}).
    
    Let $B'$ be a subset of $B$ satisfying (\ref{E50}), (\ref{E51}) and (\ref{E52}). From (\ref{E50}), we have that for any $\ab',\ab''\in B'$, $h^{(\ab)}_{\ab,\ab'}=h^{(\ab)}_{\ab,\ab''}\circ h$ for some $h$ with degree $1$; therefore $l^{(1)}_{\ab,\ab'}=h^{-1}\circ l^{(1)}_{\ab,\ab''}$ and $\deg h=\deg h^{-1}=1$; thus we can assume that all the $l^{(1)}_{\ab,\ab'}$ are equal; set $l_1:=l^{(1)}_{\ab,\ab'}$ for any $\ab'\in B'$. Likewise write $l_2:=l^{(2)}_{\ab,\ab'}$ for any $\ab'\in B'$. Let $\pi_1:A\times A\rightarrow A$ (resp. $\pi_2:A\times A\rightarrow A$) be the projection to the first (resp. second) entry. Then
    \begin{equation}
    \label{E53}
    |B'|\leq |\pi_1(B')||\pi_2(B')|;
        \end{equation}
        assume without loss of generality that $|\pi_1(B')|\geq |\pi_2(B')|$. Let $B^*$ be a subset of $B'$ such that each element in $\pi_1(B')$ has exactly one preimage in $B^*$. Then (\ref{E53}) implies that
      \begin{equation}
      \label{E54}
      |B^*|=|\pi_1(B')|\geq |B'|^{\frac{1}{2}}
      \end{equation}
        Moreover, (\ref{E52}) and (\ref{E54}) lead to
        \begin{equation}
        \label{E55}
        |B^*|\geq |B'|^{\frac{1}{2}}\geq \left(\frac{1}{2^{2d}}|B|\right)^{\frac{1}{2}}\geq \left(\frac{c_{6,d}^2}{2^{2d}}\right)^{\frac{1}{2}}=c_{3,d}.
        \end{equation}
      On the other hand, for all $\ab'=(a'_1,a'_2)\in B^*$ we have that 
      \begin{equation}
      \label{E56}
      f(y_1,a'_1)=h^{(\ab)}_{\ab,\ab'}\left(l_1(y_1)\right).
\end{equation} 
From (\ref{E55}) and (\ref{E56}), we can apply Proposition \ref{R28} ii) to $f$ and get the existence of $h\in\field(x_1,x_2)$ such that
\begin{equation*}
f(x_1,x_2)=h(l_1(x_1),x_2).
\end{equation*}  
It remains to show that $\deg_{x_1}h\leq 1$. If $\deg l_1=0$, then $f(x_1,x_2)$ does not depend on $x_1$, and thereby we can take $h(x_1,x_2)$ with $\deg_{x_1}h=0$. Thus from now on we assume that 
\begin{equation}
\label{E57}
\deg l_1>0.
\end{equation}    
 Since $|\pi_1(B')|=|B^*|\geq c_{3,d}$ by (\ref{E55}), Lemma \ref{R30} yields that 
 \begin{equation}
 \label{E58}
 \deg_{x_1}h(x_1,x_2)=\max_{a'_1\in\pi_1(B')}\deg_{x_1}h(x_1,a'_1).
 \end{equation}
 Hence
 \begin{align*}
 \deg_{x_1}h\cdot \deg l_1&=\deg_{x_1}h(l_1(x_1),x_2)\\
 &=\max_{a'_1\in\pi_1(B')}\deg_{x_1}h(l_1(x_1),a'_1)&\Big(\text{by (\ref{E58})}\Big)\\
  &=\max_{a'_1\in\pi_1(B')}\deg_{x_1}f(x_1,a'_1)\\
  &=\max_{\ab'\in B'}\deg_{x_1}h^{(\ab)}_{\ab,\ab'}\left(l_1(x_1)\right)&\Big(\text{by (\ref{E56})}\Big)\\
   &=\max_{\ab'\in B'}\deg h^{(\ab)}_{\ab,\ab'}\cdot \deg l_1\\
    &\leq \max_{\ab'\in B'} \deg l_1 &\Big(\text{since $\deg h^{(\ab)}_{\ab,\ab'}\leq 1$}\Big)\\
    &=\deg l_1,
 \end{align*}
  and finally (\ref{E57}) leads to $\deg_{x_1}h\leq 1$.
   \end{proof}
    The main step in the proof of Theorem \ref{R3} is the following lemma. 
    \begin{lem}
    \label{R35}
    Let $\field\in\{\Real, \Com\}$, $d\in\Nat$ and $f\in\field(x_1,x_2)$ be such that $\deg f\leq d$. For any subsets $A_1,A_2$  of $\field$ such that $|A_1|=|A_2|$ and $A_1\times A_2\subseteq \Dom f$, one of the following statement holds.
    \begin{enumerate}
    \item[i)]$|f(A_1,A_2)|=\Omega_d\left(|A_1|^{\frac{4}{3}}\right)$.
    \item[ii)]There are $g,l_1\in\field(x)$ and $h\in\field(x_1,x_2)$ such that $\deg_{x_1}h\leq 1$ and $f(x_1,x_2)=g(h(l_1(x_1),x_2))$.
    \item[iii)]There are $g,l_2\in\field(x)$ and $h\in\field(x_1,x_2)$ such that $\deg_{x_2}h\leq 1$ and $f(x_1,x_2)=g(h(x_1,l_2(x_2)))$.
    \end{enumerate}
    \end{lem}
    \begin{proof}
   We assume that  $|A_1|,|A_2|>\max\left\{\frac{1}{c_{7,d}}, 2\right\}$ where $c_{7,d}$ is as in Lemma \ref{R33}. If $\deg_{x_1}f=0$ (resp. $\deg_{x_2}f=0$), then we are in ii) (resp. iii)) taking $l_1(x_1)=x_1, g(x_2)=f(x_1,x_2) $ and $h(x_1,x_2)=x_2$ (resp. $l_2(x_2)=x_2, g(x_1)=f(x_1,x_2) $ and $h(x_1,x_2)=x_1$).  Thus, from now on, we assume that $\deg_{x_1} f,\deg_{x_2} f\geq 1$. From Lemma \ref{R33},  there are $\fu\in \field(x_1,x_2)$, $g\in\field(x)$,  $B_1\subseteq A_1\times A_1$ and $B_2\subseteq A_2\times A_2$ such that the following  conditions are satisfied:
\begin{enumerate}
\item[I)]$f(x_1,x_2)=g(\fu(x_1,x_2))$.
\item[II)]$|(A_1\times A_1)\setminus B_1|\leq c_{6,d}|A_1|$ and $|(A_2\times A_2)\setminus B_2|\leq c_{6,d}|A_2|$.
\item[III)]$g(x_1)$ dominates $\{f(x_1,a_2),f(x_1,a'_2)\}$ for all $(a_2,a'_2)\in B_2$ and $g(x_2)$ dominates $\{f(a_1,x_2),f(a'_1,x_2)\}$ for all $(a_1,a'_1)\in B_1$.
\end{enumerate} 
The proof of the lemma is divided into two cases.
\begin{enumerate}
\item[$\bullet$]
Assume that for any irreducible curve $C$ in $\field^2$,
\begin{equation}
\label{E59}
  \left|\left\{\ab\in B_2:\;C^{(1)}_{\fu,\ab}\supseteq  C\right\}\right |, \left|\left\{\ab\in B_1:\;C^{(2)}_{\fu,\ab}\supseteq  C\right\}\right|< c_{6,d}^2.
 \end{equation}
 From II) and (\ref{E59}), the assumptions of Theorem \ref{R23} are satisfied for $m=c_{6,d}^2$; then Theorem \ref{R23} implies that
 \begin{equation}
 \label{E60}
\left|I\left(A_1\times A_1,\Cc^{(1)}_{\fu,A_2\times A_2}\right)\right|=O_{m,d}\left(|A_1|^{\frac{8}{3}}\right)=O_{d}\left(|A_1|^{\frac{8}{3}}\right).
 \end{equation}
 Write
 \begin{equation*}
 Q:=\left\{(a_1,a_2,a'_1,a'_2)\in A_1\times A_2\times A_1\times A_2:\;\fu(a_1,a_2)=\fu(a'_1,a'_2)\right\}.
 \end{equation*}
 From the Cauchy-Schwarz inequality, 
 \begin{equation}
 \label{E61}
 \sum_{b\in \fu(A_1,A_2)}|\fu^{-1}(b)|^2\geq \frac{1}{|\fu(A_1,A_2)|}\left(\sum_{b\in \fu(A_1,A_2)}|\fu^{-1}(b)|\right)^2.
 \end{equation}
 Thus 
 \begin{align}
 \label{E62}
 |Q|&=\sum_{b\in \fu(A_1,A_2)}|\fu^{-1}(b)|^2\nonumber\\
 &\geq \frac{1}{|\fu(A_1,A_2)|}\left(\sum_{b\in \fu(A_1,A_2)}|\fu^{-1}(b)|\right)^2&\Big(\text{by (\ref{E61})}\Big)\nonumber\\
 &= \frac{(|A_1||A_2|)^2}{|\fu(A_1,A_2)|}\nonumber\\
  &= \frac{|A_1|^4}{|\fu(A_1,A_2)|}.
 \end{align}
 Now note that the definition of the algebraic curves $C^{(1)}_{\fu,(a_2,a'_2)}$ with $(a_2,a'_2)\in A_2\times A_2$ and  its points $(a_1,a'_1) \in A_1\times A_1$ yield that 
 \begin{align}
 \label{E63}
 |Q|&= \left|I\left(A_1\times A_1,\Cc^{(1)}_{\fu,A_2\times A_2}\right)\right|.
 \end{align}
 Then
 \begin{align}
 \label{E64}
 |\fu(A_1,A_2)|&\geq \frac{|A_1|^4}{|Q|}&\Big(\text{by (\ref{E62})}\Big)\nonumber\\
 &=\frac{|A_1|^4}{\left|I\left(A_1\times A_1,\Cc^{(1)}_{\fu,A_2\times A_2}\right)\right|}&\Big(\text{by (\ref{E63})}\Big)\nonumber\\
 &=\Omega_d\left(|A_1|^{\frac{4}{3}}\right).&\Big(\text{by (\ref{E60})}\Big)
 \end{align}
 Since $f(x_1,x_2)=g(\fu(x_1,x_2))$, 
 \begin{equation}
 \label{E65}
 |f(A_1,A_2)|\geq \frac{1}{\deg g}|\fu(A_1,A_2)|\geq \frac{1}{d}|\fu(A_1,A_2)|.
 \end{equation}
 Finally, (\ref{E64}) and (\ref{E65}) imply i).
  \item[$\bullet$]Assume that there exists an irreducible curve $C$ in $\field^2$ such that
\begin{equation}
\label{E66}
  \left|\left\{\ab\in B_2:\;C^{(1)}_{\fu,\ab}\supseteq  C\right\}\right | \geq  c_{6,d}^2
 \end{equation}
  or 
  \begin{equation}
  \label{E67}
  \left|\left\{\ab\in B_1:\;C^{(2)}_{\fu,\ab}\supseteq  C\right\}\right|\geq  c_{6,d}^2.
  \end{equation}
  Suppose that (\ref{E66}) happens and set $B'_2:=\left\{\ab\in B_2:\;C^{(1)}_{\fu,\ab}\supseteq  C\right\}$. Since $C$ is an irreducible curve, there is an irreducible nonconstant polynomial $r(y_1,y_2)\in\field[y_1,y_2]$ such that 
  \begin{equation*}
  C=\left\{(b_1,b_2)\in\field^2:\;r(b_1,b_2)=0\right\}.
  \end{equation*}
  For any $\ab=(a_1,a_2)\in B'_2$, write $\fu(y_1,a_1)=\frac{p_{a_1}(y_1)}{q_{a_1}(y_1)}$ and $\fu(y_2,a_2)=\frac{p_{a_2}(y_2)}{q_{a_2}(y_2)}$ with $p_{a_1},q_{a_1}\in\field[y_1]$, $p_{a_2},q_{a_2}\in\field[y_2]$, $(p_{a_1},q_{a_1})=1$ and $(p_{a_2},q_{a_2})=1$, and set $r_\ab(y_1,y_2):=p_{a_1}(y_1)q_{a_2}(y_2)-q_{a_1}(y_1)p_{a_2}(y_2)$; inasmuch as $\ab\in B'_2$, we have that $C^{(1)}_{\fu,\ab}\supseteq  C$ and this means that $r$ divides $r_\ab$.   For any $\ab=(a_1,a_2), \ab'=(a'_1,a'_2)\in B'_2$, we have that $r$ divides $r_{\ab}$ and $r_{\ab'}$. Therefore we can apply Lemma \ref{R29} to the rational functions $\fu(y_1,a_1),\fu(y_2,a_2),\fu(y_1,a'_1)$ and $\fu(y_2,a'_2)$, and we get there exist $h^{(\ab)}_{\ab,\ab'},h^{(\ab')}_{\ab,\ab'},l^{(1)}_{\ab,\ab'},l^{(2)}_{\ab,\ab'}\in\field(x)$ such that
  \begin{align}
  \label{E68}
      \fu(y_1,a_1)&=h^{(\ab)}_{\ab,\ab'}\left(l^{(1)}_{\ab,\ab'}(y_1)\right)& \fu(y_2,a_2)&=h^{(\ab)}_{\ab,\ab'}\left(l^{(2)}_{\ab,\ab'}(y_2)\right)\nonumber\\
    \fu(y_1,a'_1)&=h^{(\ab')}_{\ab,\ab'}\left(l^{(1)}_{\ab,\ab'}(y_1)\right)& \fu(y_2,a'_2)&=h^{(\ab')}_{\ab,\ab'}\left(l^{(2)}_{\ab,\ab'}(y_2)\right).
     \end{align}
     On the one hand, (\ref{E66}) implies that $|B'_2|\geq c_{6,d}^2$. On the other hand, since $\fu(x_1,a_1)$ and $\fu(x_1,a_2)$ factorize through $h^{(\ab)}_{\ab,\ab'}$ by (\ref{E68}), we see that $f(x_1,a_1)$ and $f(x_1,a_2)$ factorize through $g\circ h^{(\ab)}_{\ab,\ab'}$; however $\ab\in B_2$ so $g$ dominates $\{f(x_1,a_1),f(x_1,a_2)\}$ and therefore $\deg g\geq \deg g\circ h^{(\ab)}_{\ab,\ab'}$; in particular $\deg h^{(\ab)}_{\ab,\ab'}\leq 1$ for all $\ab,\ab'\in B'_2$. Therefore the assumptions of Lemma \ref{R34} i) are satisfied, and thereby we can conclude the existence of $h\in\field(x_1,x_2)$ and $l_1\in\field(x)$ such that $\deg_{x_1}h\leq 1$ and 
     \begin{equation}
     \label{E69}
     \fu(x_1,x_2)=h(l_1(x_1),x_2).
     \end{equation}
     From I) and (\ref{E69}), we get ii). Likewise, if (\ref{E67}) holds, we obtain iii).\qedhere
 \end{enumerate}
  \end{proof}
  
    Finally, we show Theorem \ref{R3}.
    \begin{proof}
   \emph{(Theorem \ref{R3}) }Assume that $|A_1|,|A_2|>d\max\left\{\frac{1}{c_{7,d}},2\right\}$ with $c_{7,d}$ as in Lemma \ref{R33}. From Lemma \ref{R35}, we have $3$ possibilities. If Lemma \ref{R35} i) holds, then we are done. It remains to prove Theorem \ref{R3} when  Lemma \ref{R35} ii) or Lemma \ref{R35} iii) happen but they are symmetric; thus we prove our theorem when  Lemma \ref{R35} ii) occurs and the other case is proven \emph{mutatis mutandis}. Then there are $\gu,\lu\in\field(x)$ and $\fu\in\field(x_1,x_2)$ such that $f(x_1,x_2)=\gu(\fu(\lu(x_1),x_2))$ and $\deg_{x_1} \fu\leq 1$. If $\deg_{x_1} \fu=0$, then $\fu(x_1,x_2)=\tilde{h}(x_2)$ for some $\tilde{h}\in\field(x)$, and ii) holds in this case. From now on, we assume that
   \begin{equation}
   \label{E70}
   \deg_{x_1}\fu=1.
   \end{equation}
   Write $A'_1:=\lu(A_1)$ and let $A'_2$ be a subset of $A_2$ such that $|A'_2|=|A'_1|$. Note that
   \begin{equation}
   \label{E71}
   |A'_2|=|A'_1|=|\lu(A_1)|\geq \frac{1}{\deg \lu}|A_1|\geq \frac{1}{d}|A_1|.
   \end{equation}
  Since $|A_1|,|A_2|>d\max\left\{\frac{1}{c_{7,d}}, 2\right\}$, we have by (\ref{E71}) that $|A'_1|,|A'_2|>\max\left\{\frac{1}{c_{7,d}}, 2\right\}$. Thus, from Lemma \ref{R33}, there are $\fd\in \field(x_1,x_2)$, $\gd\in\field(x)$,  $B'_1\subseteq A'_1\times A'_1$ and $B'_2\subseteq A'_2\times A'_2$ such that the following conditions are satisfied.
\begin{enumerate}
\item[I)]$\fu(x_1,x_2)=\gd(\fd(x_1,x_2))$.
\item[II)]$|(A'_1\times A'_1)\setminus B'_1|\leq c_{6,d}|A'_1|$ and $|(A'_2\times A'_2)\setminus B'_2|\leq c_{6,d}|A'_2|$.
\item[III)]$\gd(x_1)$ dominates $\{\fu(x_1,a_2),\fu(x_1,a'_2)\}$ for all $(a_2,a'_2)\in B'_2$ and $\gd(x_2)$ dominates $\{\fu(a_1,x_2),\fu(a'_1,x_2)\}$ for all $(a_1,a'_1)\in B'_1$.
\end{enumerate} 
Since $\deg_{x_1} \fu=1$, I) implies that $\deg \gd=1$ and therefore we may assume that $\gd(x)=x$ from now on; therefore $\fu=\fd$. Set 
\begin{equation*}
B''_2:=\left\{(a,a')\in B'_2:\;\deg \fu(x_1,a)=\deg \fu(x_1,a')=1\right\}.
\end{equation*}
From (\ref{E70}), we know that there exist $p_1,p_0,q_1,q_0\in\field[x_2]$ such that $\fu(x_1,x_2)=\frac{p_1(x_2)x_1+p_0(x_2)}{q_1(x_2)x_1+q_0(x_2)}$. Thus, for all $a\in \field$ such that $a$ is neither a root of $p_1$ nor a root of $q_1$, we have that $\deg \fu(x_1,a)=1$. Since $\deg p_1,\deg q_1\leq d$, the previous discussion yields that 
\begin{equation}
\label{E72}
|B''_2|\geq |B'_2|-4d|A'_2|.
\end{equation}
The proof is divided into two cases.
\begin{enumerate}
\item[$\bullet$] Assume that for any irreducible curve $C$ in $\field^2$,
\begin{equation}
\label{E73}
  \left|\left\{\ab\in B'_1:\;C^{(2)}_{\fu,\ab}\supseteq  C\right\}\right|< c_{6,d}^2.
 \end{equation}
For all $(a,a')\in B''_2$, since $\deg \fu(x_1,a)=\deg \fu(x_1,a')=1$, Lemma \ref{R17} implies that the algebraic curve $C^{(1)}_{\fu,(a,a')}$ is irreducible and thereby 
 \begin{equation}
\label{E74}
  \left|\left\{\ab\in B''_2:\;C^{(1)}_{\fu,\ab}\supseteq  C\right\}\right|\leq 1\leq  c_{6,d}^2.
 \end{equation}
 From II), (\ref{E72}),(\ref{E73}) and (\ref{E74}), the assumptions of Theorem \ref{R23} are satisfied for $m=c_{6,d}^2$; then Theorem \ref{R23} implies that
 \begin{equation}
 \label{E75}
\left|I\left(A'_1\times A'_1,\Cc^{(1)}_{\fu,A'_2\times A'_2}\right)\right|=O_{m,d}\left(|A'_1|^{\frac{8}{3}}\right)=O_{d}\left(|A'_1|^{\frac{8}{3}}\right).
 \end{equation}
 Set 
  \begin{equation*}
 Q:=\left\{(a_1,a_2,a'_1,a'_2)\in A'_1\times A'_2\times A'_1\times A'_2:\;\fu(a_1,a_2)=\fu(a'_1,a'_2)\right\}.
 \end{equation*}
 From the Cauchy-Schwarz inequality, 
 \begin{equation}
 \label{E76}
 \sum_{b\in \fu(A'_1,A'_2)}|\fu^{-1}(b)|^2\geq \frac{1}{|\fu(A'_1,A'_2)|}\left(\sum_{b\in \fu(A'_1,A'_2)}|\fu^{-1}(b)|\right)^2.
 \end{equation}
 Hence 
 \begin{align}
 \label{E77}
 |Q|&=\sum_{b\in \fu(A'_1,A'_2)}|\fu^{-1}(b)|^2\nonumber\\
 &\geq \frac{1}{|\fu(A'_1,A'_2)|}\left(\sum_{b\in \fu(A'_1,A'_2)}|\fu^{-1}(b)|\right)^2&\Big(\text{by (\ref{E76})}\Big)\nonumber\\
 &= \frac{(|A'_1||A'_2|)^2}{|\fu(A'_1,A'_2)|}\nonumber\\
  &= \frac{|A'_1|^4}{|\fu(A'_1,A'_2)|}.
 \end{align}
 Now note that the definition of the algebraic curves $C^{(1)}_{\fu,(a_2,a'_2)}$ with $(a_2,a'_2)\in A'_2\times A'_2$ and  its points $(a_1,a'_1) \in A'_1\times A'_1$ yield that 
 \begin{align}
 \label{E78}
 |Q|&= \left|I\left(A'_1\times A'_1,\Cc^{(1)}_{\fu,A'_2\times A'_2}\right)\right|.
 \end{align}
 Then
 \begin{align}
 \label{E79}
 |\fu(A'_1,A'_2)|&\geq \frac{|A'_1|^4}{|Q|}&\Big(\text{by (\ref{E77})}\Big)\nonumber\\
 &=\frac{|A'_1|^4}{\left|I\left(A'_1\times A'_1,\Cc^{(1)}_{\fu,A'_2\times A'_2}\right)\right|}&\Big(\text{by (\ref{E78})}\Big)\nonumber\\
 &=\Omega_d\left(|A'_1|^{\frac{4}{3}}\right).&\Big(\text{by (\ref{E75})}\Big)
 \end{align}
 Since $A'_2\subseteq A_2, A'_1=\lu(A_1)$ and $f(x_1,x_2)=\gu(\fu(\lu(x_1),x_2))$, we obtain that
 \begin{equation}
 \label{E80}
  |f(A_1,A_2)|\geq |\fu(A'_1,A'_2)|.
 \end{equation}
 Finally, 
 \begin{align*}
  |f(A_1,A_2)|&\geq |\fu(A'_1,A'_2)|&\Big(\text{by (\ref{E80})}\Big)\\
  &=\Omega_d\left(|A'_1|^{\frac{4}{3}}\right)&\Big(\text{by (\ref{E79})}\Big)\\
    &=\Omega_d\left(|A_1|^{\frac{4}{3}}\right).&\Big(\text{by (\ref{E71})}\Big)
 \end{align*}
    \item[$\bullet$] Assume that there exists an irreducible curve $C$ in $\field^2$ such that
\begin{equation}
\label{E81}
  \left|\left\{\ab\in B'_1:\;C^{(2)}_{\fu,\ab}\supseteq  C\right\}\right|\geq  c_{6,d}^2.
 \end{equation}
   Set $B''_1:=\left\{\ab\in B'_1:\;C^{(2)}_{\fu,\ab}\supseteq  C\right\}$. Since $C$ is an irreducible curve, there is an irreducible nonconstant polynomial $r(y_1,y_2)\in\field[y_1,y_2]$ such that 
  \begin{equation*}
  C=\left\{(b_1,b_2)\in\field^2:\;r(b_1,b_2)=0\right\}.
  \end{equation*}
  For any $\ab=(a_1,a_2)\in B''_1$, write $\fu(a_1,y_1)=\frac{p_{a_1}(y_1)}{q_{a_1}(y_1)}$ and $\fu(a_2,y_2)=\frac{p_{a_2}(y_2)}{q_{a_2}(y_2)}$ with $p_{a_1},q_{a_1}\in\field[y_1]$, $p_{a_2},q_{a_2}\in\field[y_2]$, $(p_{a_1},q_{a_1})=1$ and $(p_{a_2},q_{a_2})=1$, and set $r_\ab(y_1,y_2):=p_{a_1}(y_1)q_{a_2}(y_2)-q_{a_1}(y_1)p_{a_2}(y_2)$; inasmuch as $\ab\in B''_1$, we have that $C^{(2)}_{\fu,\ab}\supseteq  C$ and this means that $r$ divides $r_\ab$.   For any $\ab=(a_1,a_2), \ab'=(a'_1,a'_2)\in B''_1$, we have that $r$ divides $r_{\ab}$ and $r_{\ab'}$. Therefore we can apply Lemma \ref{R29} to the rational functions $\fu(a_1,y_1),\fu(a_2,y_2),\fu(a'_1,y_1)$ and $\fu(a'_2,y_2)$, and we get there exist $h^{(\ab)}_{\ab,\ab'},h^{(\ab')}_{\ab,\ab'},l^{(1)}_{\ab,\ab'},l^{(2)}_{\ab,\ab'}\in\field(x)$ such that
  \begin{align}
  \label{E82}
      \fu(a_1,y_1)&=h^{(\ab)}_{\ab,\ab'}\left(l^{(1)}_{\ab,\ab'}(y_1)\right)& \fu(a_2,y_2)&=h^{(\ab)}_{\ab,\ab'}\left(l^{(2)}_{\ab,\ab'}(y_2)\right)\nonumber\\
    \fu(a'_1,y_1)&=h^{(\ab')}_{\ab,\ab'}\left(l^{(1)}_{\ab,\ab'}(y_1)\right)& \fu(a'_2,y_2)&=h^{(\ab')}_{\ab,\ab'}\left(l^{(2)}_{\ab,\ab'}(y_2)\right).
     \end{align}
     On the one hand, (\ref{E81}) implies that $|B''_1|\geq c_{6,d}^2$. On the other hand, since $\fu(a_1,x_2)$ and $\fu(a_2,x_2)$ factorize through $h^{(\ab)}_{\ab,\ab'}$ by (\ref{E82}), we see that $f(a_1,x_2)$ and $f(a_2,x_2)$ factorize through $g\circ h^{(\ab)}_{\ab,\ab'}$; however $\ab\in B'_1$ so $g$ dominates $\{f(a_1,x_2),f(a_2,x_2)\}$ and therefore $\deg g\geq \deg g\circ h^{(\ab)}_{\ab,\ab'}$; in particular $\deg h^{(\ab)}_{\ab,\ab'}\leq 1$ for all $\ab,\ab'\in B''_1$. Therefore the assumptions of Lemma \ref{R34} ii) are satisfied, and we can conclude the existence of $h\in\field(x_1,x_2)$ and $l_2\in\field(x)$ such that $\deg_{x_2}h\leq 1$ and 
     \begin{equation}
     \label{E83}
     \fu(x_1,x_2)=h(x_1,l(x_2)).
     \end{equation}
   Since $f(x_1,x_2)=g(\fu (l_1(x_1),x_2))$, we get from (\ref{E83}) that
   \begin{equation*}
   f(x_1,x_2)=g(h(l_1(x_1),l_2(x_2))).\qedhere
\end{equation*}    
 \end{enumerate}
 \end{proof}
  \section{Proofs of the main results}
    In this section we conclude the proofs of the main results of this paper. Before we complete the proof of these results, we will need some trivial but important facts. Let $\field\in\{\Real,\Com\}$, $d\in\Nat$, $g\in\field(x)$ and $f,h\in\field(x_1,x_2)$ be such that $\deg f\leq d$ and $f(x_1,x_2)=g(h(x_1,x_2))$. Then $\deg g\leq \deg f\leq d$. Hence, for any $A_1, A_2$ nonempty finite subsets of $\field$ such that $A_1\times A_2\subseteq \Dom f$,
    \begin{equation*}
    |h(A_1,A_2)|\geq |f(A_1,A_2)|\geq \frac{1}{d}|h(A_1,A_2)|
    \end{equation*}
    This claim leads to the following fact.
    \begin{rem}
    \label{R36}
     Let $\field\in\{\Real,\Com\}$, $d\in\Nat$, $g\in\field(x)$ and $h\in\field(x_1,x_2)$ be such that $\deg g(h(x_1,x_2))\leq d$. For any $a\in\field$, set $h_a(x_1,x_2):=h(x_1,x_2+a)$  and $f_a(x_1,x_2):= g(h_a(x_1,x_2))$. Then, for any $a,b\in \field$ and $A_1, A_2$ nonempty finite subsets of $\field$ such that $|A_1|=|A_2|$ and $A_1\times A_2\subseteq \Dom f_0$, 
     \begin{equation*}
     |f_a(A_1,A_2-\{a\})|=\Omega_d\left(|A_1|^{\frac{4}{3}}\right)
\end{equation*}      
if and only if 
\begin{equation*}
     |f_b(A_1,A_2-\{b\})|=\Omega_d\left(|A_1|^{\frac{4}{3}}\right).
\end{equation*} 
    \end{rem}
   For any $\field\in\{\Real,\Com\}$, denote by $\Mat(\field)$ the set of $2\times 2$-matrices with coefficients in $\field$, and by $\GL(\field)$ the invertible elements of $\Mat(\field)$.  For any $a\in \field$ and $X\in\Mat(\field)$, we denote by $aX=a\cdot X$ the scalar multiplication of  $X$ by $a$. Another important fact, and the reason why we have different characterizations of the nonexpander rational functions in the real and the  complex case, is the Jordan decomposition of the elements of $\Mat(\field)$.
    \begin{rem}
    \label{R37}
    \begin{enumerate}
    \item[i)]Let $X\in\Mat(\Real)$. Then there are $H\in \GL(\Real)$ and $J\in\Mat(\Real)$ with one of the following forms
    \begin{enumerate}
     \item[I)]\begin{equation*}
J= \left( \begin{array}{ccc}
a&1\\
0&a
 \end{array} \right),\qquad a\in\Real
  \end{equation*}
    \item[II)]\begin{equation*}
J= \left( \begin{array}{ccc}
a&0\\
0&b
 \end{array} \right),\qquad a,b\in\Real
  \end{equation*}
   \item[III)]\begin{equation*}
J= \left( \begin{array}{ccc}
a&-b\\
b&a
 \end{array} \right),\qquad a,b\in\Real
  \end{equation*}
    \end{enumerate}
    such that $X=HJH^{-1}$.
     \item[ii)]Let $X\in\Mat(\Com)$. Then there are $H\in \GL(\Com)$ and $J\in\Mat(\Com)$ with one of the following forms
    \begin{enumerate}
     \item[I)]\begin{equation*}
J= \left( \begin{array}{ccc}
a&1\\
0&a
 \end{array} \right),\qquad a\in\Com
  \end{equation*}
    \item[II)]\begin{equation*}
J= \left( \begin{array}{ccc}
a&0\\
0&b
 \end{array} \right),\qquad a,b\in\Com
  \end{equation*}
   \end{enumerate}
    such that $X=HJH^{-1}$.
    \end{enumerate}
        \end{rem}
      For each $\field\in\{\Real,\Com\}$ and  $X=\left( \begin{array}{ccc}
a_1&a_2\\
a_3&a_4
 \end{array} \right)\in \Mat(\field)$, define the rational function  $g_X(x)=\frac{a_1x+a_2}{a_3x+a_4}$.  With a simple calculation, we get the following fact. 
 \begin{rem}
 \label{R38}
 Let  $\field\in\{\Real,\Com\}$ and  $X,Y\in \Mat(\field)$. Then
 \begin{equation*}
 g_{XY}(x)=g_X(g_Y(x)).
 \end{equation*}
 \end{rem}
 The last remark that we need is a consequence of the fact that one variable polynomials of degree at most $d$ cannot have more than $d$ roots.
 \begin{rem}
 \label{R39}
 Let $\field\in\{\Real,\Com\}$, $d\in\Nat$, $A$ be a subset of $\field$ with $|A|>2d$ and $f_1,f_2\in\field(x_1,x_2)$ be such that $\deg f_1,\deg f_2\leq d$. If $f_1(x_1,a)=f_2(x_1,a)$ for all $a\in A$, then $f_1=f_2$.
 \end{rem}
 We conclude the proof of Theorem \ref{R4}.
 \begin{proof}(\emph{Theorem \ref{R4}}) From Theorem \ref{R3}, we have two possibilities. If $|f(A_1,A_2)|=\Omega_d\left(|A_1|^{\frac{4}{3}}\right)$, there is nothing to prove.  Thus it remains to prove Theorem \ref{R4} when there are $\gd,\ld,\ldd\in\Real(x)$ and $\fd\in\Real(x_1,x_2)$ such that $f(x_1,x_2)=\gd(\fd(\ld(x_1),\ldd(x_2)))$ and $\deg_{x_1}\fd,\deg_{x_2}\fd\leq 1$. Therefore there are $a_1,a_2,a_3,a_4,b_1,b_2,b_3,b_4\in\Real$ such that
 \begin{equation*}
\fd(x_1,x_2)=\frac{a_1x_1x_2+a_2x_2+b_1x_1+b_2}{a_3x_1x_2+a_4x_2+b_3x_1+b_4}
 \end{equation*}
 If $\deg_{x_2}\fd=0$, then $f$ is independent of $x_1$ and we can write it as in i) or ii) (i.e. $f(x_1,x_2)=\hat{f}(x_1)$ for some rational function $\hat{f_1}$). Thus we assume that $\deg_{x_2} \fd=1$ and likewise we can assume that  $\deg_{x_1} \fd=1$. In particular that means that  there exists $a\in\Real$ such that $\fd(x_1,a)$ is not a constant rational  function in $\field(x_1)$. From Remark \ref{R36}, we may assume that $a=0$ so $\fd(x_1,0)$ is not a constant rational  function in $\Real(x_1)$. Set
 \begin{equation*}
 X:=\left( \begin{array}{ccc}
a_1&a_2\\
a_3&a_4
 \end{array} \right),\quad  Y:=\left( \begin{array}{ccc}
b_1&b_2\\
b_3&b_4
 \end{array} \right).
\end{equation*}  
Since $\fd(x_1,0)$ is not constant, $Y$ is invertible and we define $Z:=Y^{-1}X$; also denote by $I$ the identity matrix in $\Mat(\Real)$. For any $e\in\Real$,
\begin{equation}
\label{E84}
eX+Y=Y(eZ+I).
\end{equation}
There are $H\in\GL(\Real)$ and $J\in\Mat(\Real)$ such that $J$ has one of the forms of Remark \ref{R37} i) and $Z=HJH^{-1}$. For any $e\in \Real$, 
\begin{align}
\label{E85}
eX+Y&=Y(eZ+I)&\Big(\text{by (\ref{E84})}\Big)\nonumber\\
&=Y\left(eHJH^{-1}+HH^{1}\right)\nonumber\\
&=YH(eJ+I)H^{-1};
\end{align}
 Remark \ref{R38} and (\ref{E85}) lead to 
\begin{equation}
\label{E86}
\fd(x_1,e)=g_{eX+Y}(x_1)=g_Y\circ g_H\circ g_{eJ+I}\circ g_{H^{-1}}(x_1).
\end{equation}
We conclude the proof depending on which case of Remark \ref{R37} i) we are. 
\begin{enumerate}
\item[I)]Assume that $J= \left( \begin{array}{ccc}
a&1\\
0&a
 \end{array} \right)$ for some $a\in\Real$. Then 
 \begin{equation*}
 g_{eJ+I}(x_1)=\frac{(ae+1)x_1+e}{ae+1}=x+\frac{e}{ae+1}
 \end{equation*}
 Set $\hat{l_2}(x_2)=\frac{x_2}{ax_2+1}\in\Real(x_2)$. Taking  $g=\gd\circ g_Y\circ g_H$, $l_1=g_{H^{-1}}\circ \ld$, $l_2=\hat{l_2}\circ \ldd$, we get that for all $e\in \Real$, 
 \begin{equation*}
 f(x_1,e)=g(l_1(x_1)+l_2(e)),
 \end{equation*}
 and Remark \ref{R39} implies that $f(x_1,x_2)=g(l_1(x_1)+l_2(x_2))$.
 \item[II)]Assume that $J= \left( \begin{array}{ccc}
a&0\\
0&b
 \end{array} \right)$ for some $a,b\in\Real$. Then 
 \begin{equation*}
 g_{eJ+I}(x_1)=\frac{(ae+1)x_1}{be+1}=x_1\cdot \left(\frac{ae+1}{be+1}\right)
 \end{equation*}
 Set $\hat{l_2}(x_2)=\frac{ax_2+1}{bx_2+1}\in\Real(x_2)$. Taking  $g=\gd\circ g_Y\circ g_H$, $l_1=g_{H^{-1}}\circ \ld$, $l_2=\hat{l_2}\circ \ldd$, we get that for all $e\in \Real$, 
 \begin{equation*}
 f(x_1,e)=g(l_1(x_1)\cdot l_2(e)),
 \end{equation*}
 and Remark \ref{R39} implies that $f(x_1,x_2)=g(l_1(x_1)\cdot l_2(x_2))$.
 \item[III)]Assume that $J= \left( \begin{array}{ccc}
a&-b\\
b&a
 \end{array} \right)$ for some $a,b\in\Real$. Then 
 \begin{equation*}
 g_{eJ+I}(x_1)=\frac{(ae+1)x_1-be}{bex_1+(ae+1)}=\frac{x_1+ \frac{be}{ae+1}}{- \frac{be}{ae+1}\cdot x_1+1}
 \end{equation*}
 Set $\hat{l_2}(x_2)=\frac{bx_2}{ax_2+1}\in\Real(x_2)$. Taking  $g=\gd\circ g_Y\circ g_H$, $l_1=g_{H^{-1}}\circ \ld$, $l_2=\hat{l_2}\circ \ldd$, we get that for all $e\in \Real$, 
 \begin{equation*}
 f(x_1,e)=g\left(\frac{l_1(x_1)+l_2(e)}{1-l_1(x_1)l_2(e)}\right),
 \end{equation*}
 and Remark \ref{R39} implies that $f(x_1,x_2)=g\left(\frac{l_1(x_1)+l_2(x_2)}{1-l_1(x_1)l_2(x_2)}\right)$.\qedhere
\end{enumerate}
 \end{proof}
 
 The proof of Theorem \ref{R5} is the same proof that we had in Theorem \ref{R4} but instead of using Remark \ref{R37} i), we use Remark \ref{R37} ii).
   
\end{document}